\documentclass[reqno]{amsart}

\usepackage{amsmath,amssymb,amsthm}
\usepackage{graphicx,tikz}
\newtheorem{theorem}{Theorem}

\newtheorem{proposition}{Proposition}

\theoremstyle{definition}

\theoremstyle{remark}

\DeclareMathOperator{\diam}{diam}

\begin{document}

\title[]{The Boundary of a Graph and \\its Isoperimetric Inequality}
\thanks{S.S. is supported by the NSF (DMS-2123224) and the Alfred P. Sloan Foundation.}

\author[]{Stefan Steinerberger}
\address{Department of Mathematics, University of Washington, Seattle, WA 98195, USA}
\email{steinerb@uw.edu}

\begin{abstract} We define, for any graph $G=(V,E)$, a boundary $\partial G \subseteq V$. The definition coincides
with what one would expected for the discretization of (sufficiently nice) Euclidean domains and contains all vertices
from the Chartrand, Erwin, Johns \& Zhang boundary. Moreover, it satisfies an isoperimetric principle stating that
graphs with many vertices have a large boundary unless they contain long paths: we show that for graphs with maximal degree $\Delta$
 $$ | \partial G| \geq \frac{1}{2\Delta}\hspace{1pt} \frac{|V|}{\diam(G)}.$$
For graphs discretizing Euclidean domains, one has $\diam(G) \sim |V|^{1/d}$ and recovers the scaling of the classical Euclidean
 isoperimetric principle. 
\end{abstract}

\maketitle

\vspace{0pt}
\section{Introduction}
We define a notion of boundary $\partial G \subseteq V$ for a graph $G=(V,E)$. We first discuss what one would expect from such a notion: when thinking
of `generic' subsets $\Omega \subset \mathbb{R}^d$ one realizes that there are strong conditions on $\Omega$ that are needed to meaningfully
define a boundary (see geometric measure theory, e.g. Federer \cite{federer}). Since graphs can `exhibit' hyperbolic structure (for example Erd\H{o}s-Renyi random graphs or expanders), most vertices should be boundary vertices unless the graph exhibits special structure. Another fundamental fact regarding the notion of boundary is that a `large' set should have a `large' boundary. For domains $\Omega \subset \mathbb{R}^d$ this can be formulated in terms of the isoperimetric inequality: it states that $\mbox{area}(\partial \Omega) \geq c \cdot \mbox{vol}(\Omega)^{(d-1)/d}$ with equality attained only when $\Omega$ is a ball. A similar principle should be true for $\partial G$: graphs with many vertices should have a large boundary -- however, here things already get more complicated because a path graph clearly has two boundary vertices independently of length.

\begin{center}
\begin{figure}[h!]
\begin{tikzpicture}[scale=0.8]
\node at (0,0) {\includegraphics[width=0.18\textwidth]{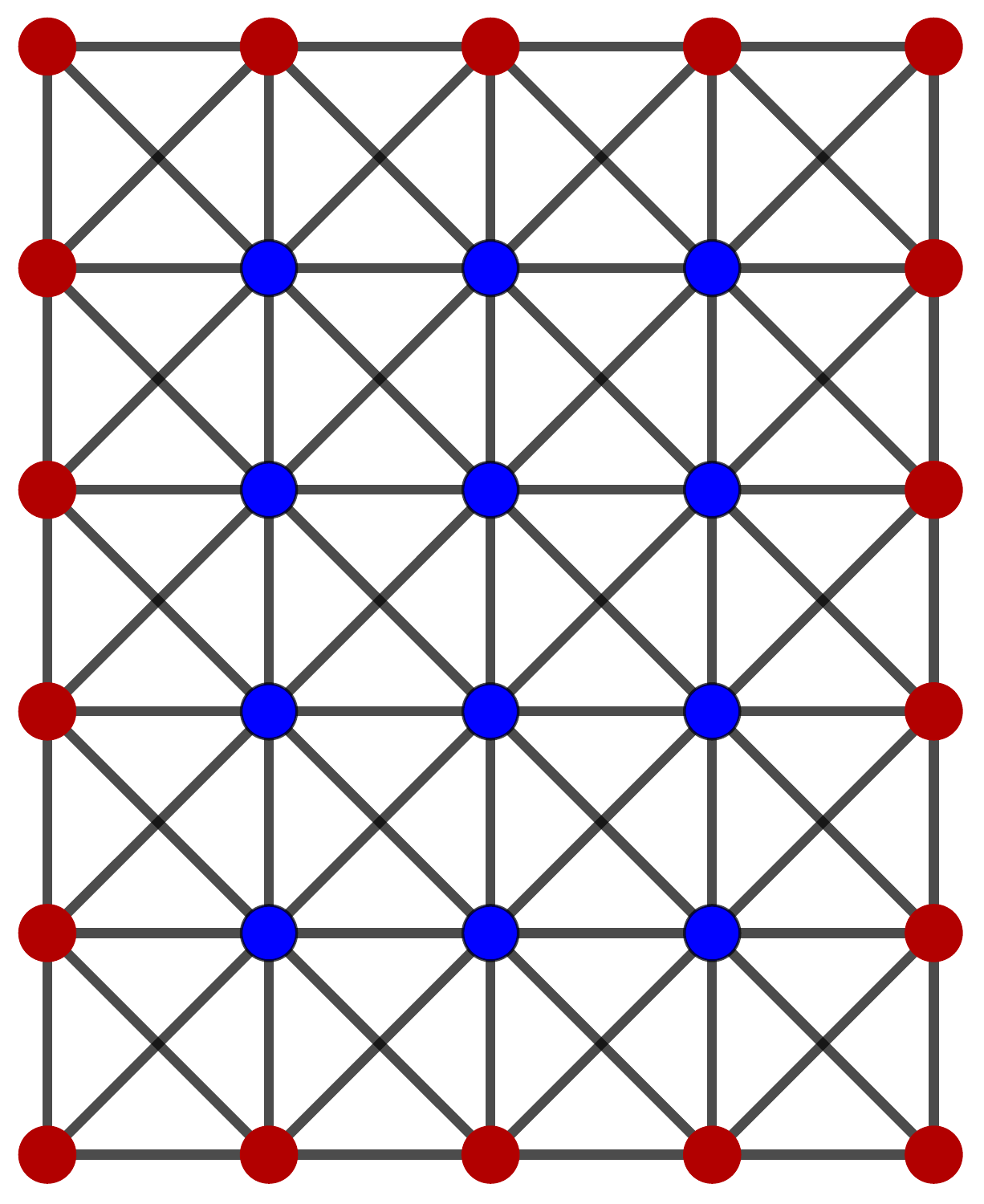}};
\node at (4.2,0) {\includegraphics[width=0.21\textwidth]{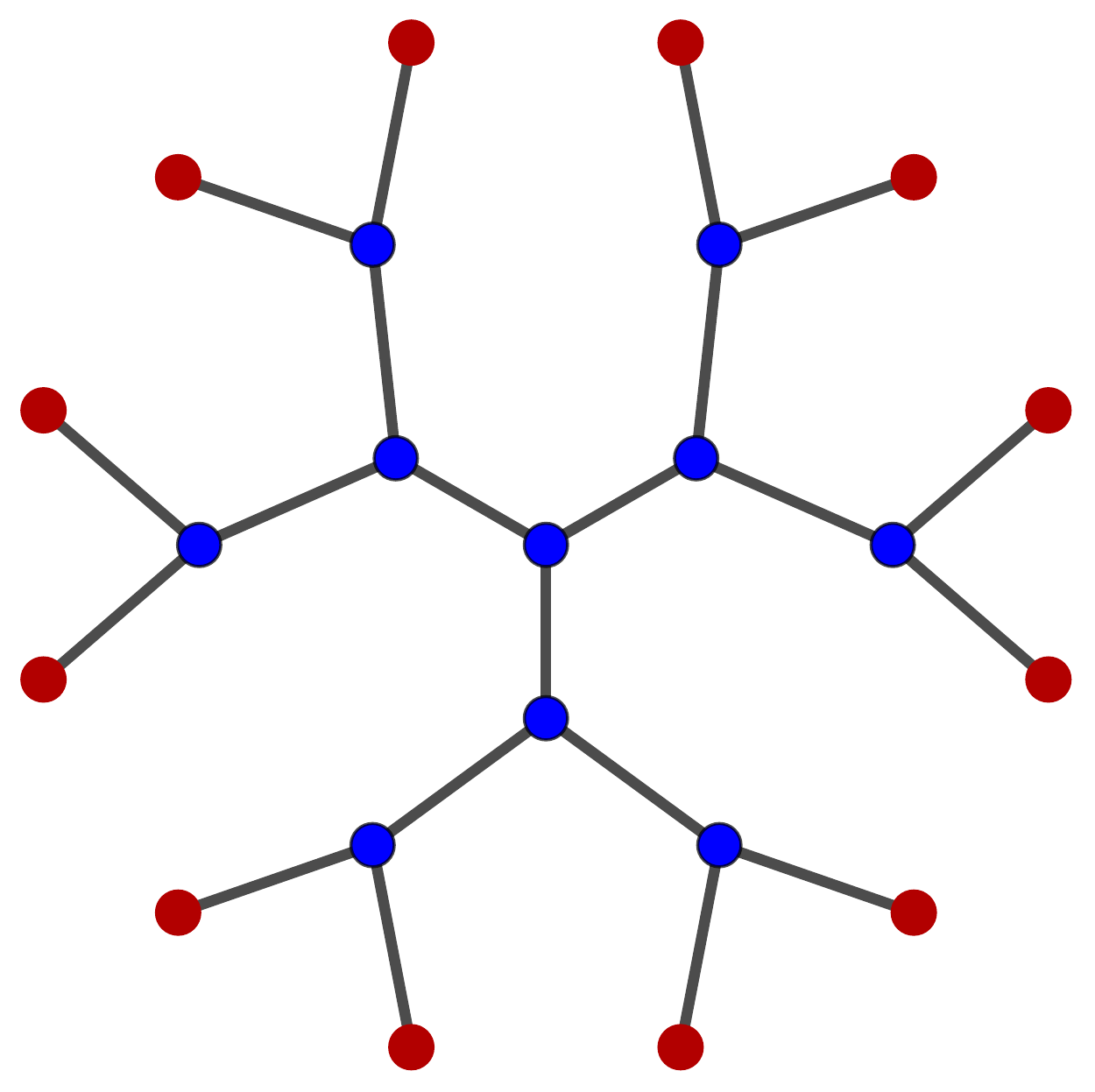}};
\node at (9.5,0) {\includegraphics[width=0.34\textwidth]{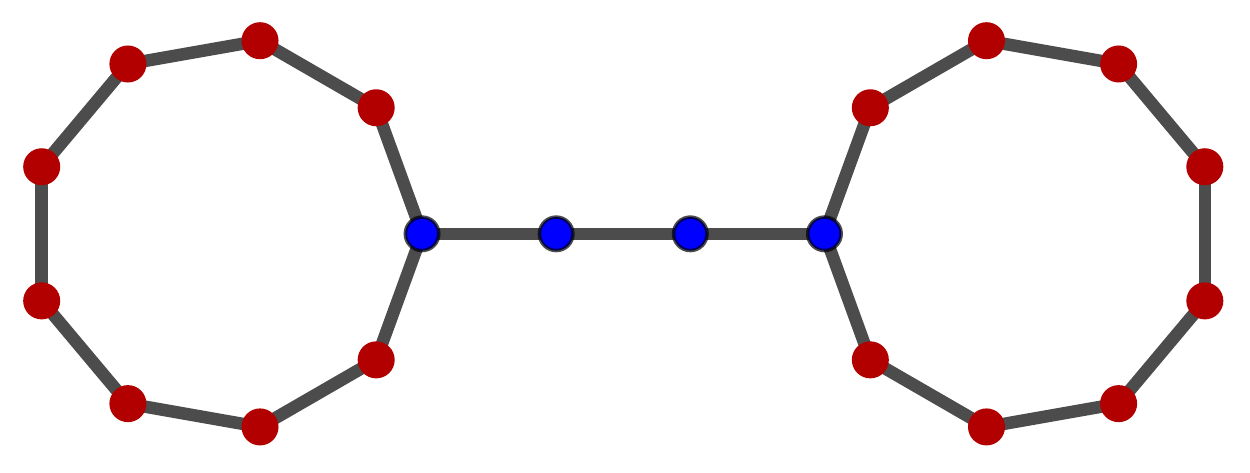}};
\end{tikzpicture}
\vspace{-10pt}
\caption{Graphs, their boundary $\partial G$ (red) and interior $V \setminus \partial G$ (blue).}
\end{figure}
\end{center}

\textbf{The Chartrand, Erwin, Johns \& Zhang boundary.} A definition of a boundary, which we will denote by $(\partial G)^* \subseteq V$ throughout the paper,  was given by Chartrand, Erwin, Johns \& Zhang \cite{ch}. The idea is very natural: we say that a vertex $u \in V$ is a boundary vertex, $u \in (\partial G)^*$, if there exists another vertex $w \in V$ such that the neighbors of $u$ are all not further away from $w$ than $u$ (see Fig. 2).
\begin{center}
\begin{figure}[h!]
\begin{tikzpicture}
\filldraw (0,0) circle (0.06cm);
\filldraw (1,0) circle (0.06cm);
\filldraw (2,0) circle (0.06cm);
\draw [thick] (0,0) -- (1,0);
\draw[thick]  (1,0) -- (2,0); 
\node at (0,-0.3) {$w$};
\node at (2,-0.3) {$u$};
\filldraw (1,1) circle (0.06cm);
 \filldraw (1,-1) circle (0.06cm);
 \draw [thick] (1,-1) -- (1,1) -- (2,0) -- (1,-1);
\end{tikzpicture}
\caption{$u$ is a boundary vertex because $d(w,u) = 2$ and every neighbor of $u$ is at most distance 2 from $w$.}
\end{figure}
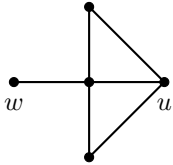
\end{center}
\vspace{-15pt}
Formally, we can define the Chartrand-Erwin-Johns-Zhang boundary as
$$(\partial G)^* = \left\{u \in V \big|  ~\exists v \in V ~ \forall (u,w) \in E: \quad d(w,v) \leq d(u,v)  \right\}.$$
Examples of the definition are shown in Fig. 1 (these examples are such that $(\partial G)^*$ coincides with our definition of boundary $\partial G$). $(\partial G)^*$ has been studied in a variety of papers, we refer to Caceres-Hernando-Mora-Pelayo-Puertas-Seara \cite{ce},  Chartrand-Erwin-Johns-Zhang \cite{ch2}, Hasegawa-Saito \cite{has},  Hernando-Mora-Pelayo-Seara \cite{her} and  M\"uller-P\'or-Sereni \cite{mu, mu2}. We also refer to \cite{all, art, bela, eroh, harper, mezz, pel} for related results.

\begin{center}
\begin{figure}[h!]
\begin{tikzpicture}
\node at (0,0) {\includegraphics[width=0.35\textwidth]{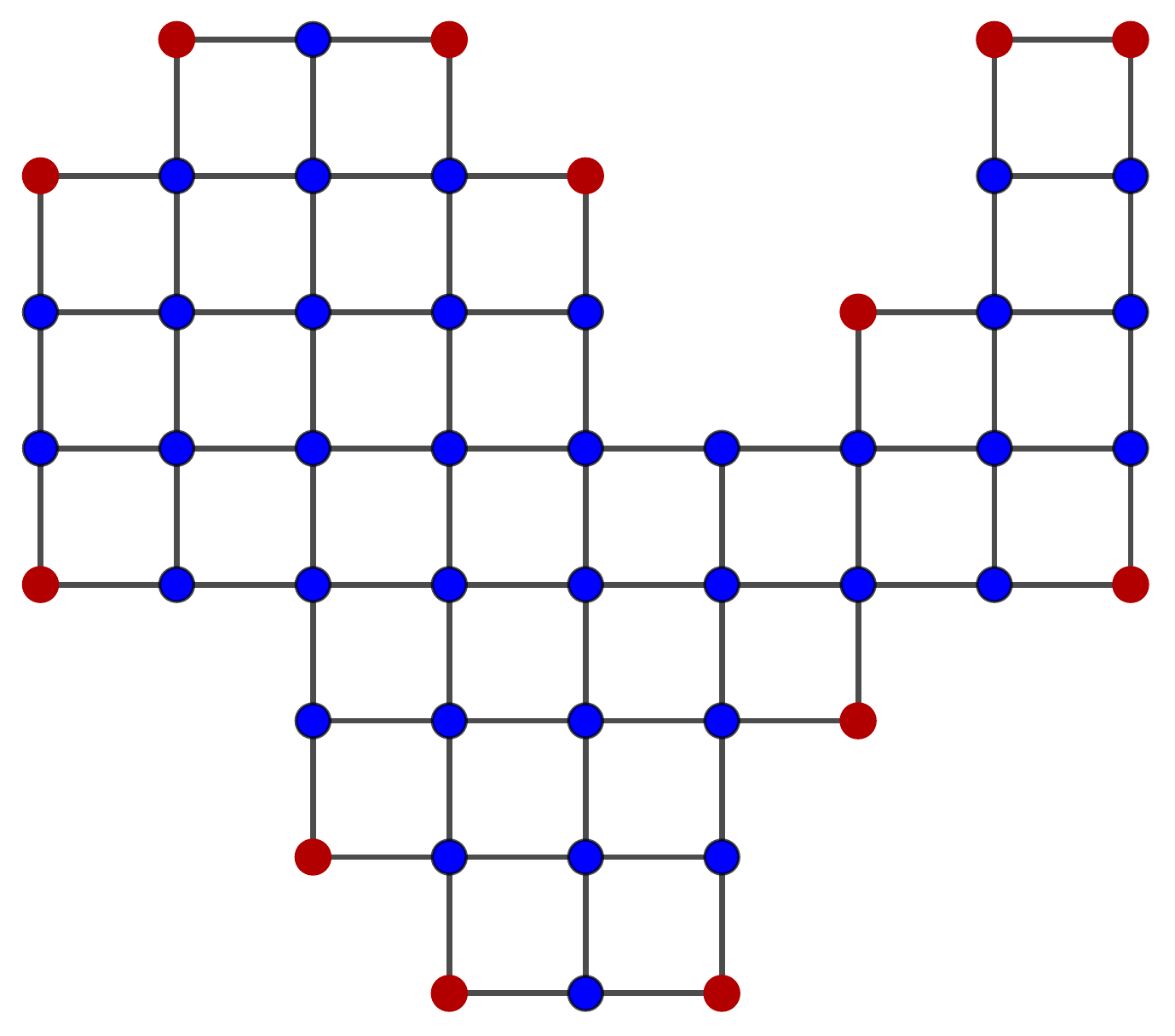}};
\node at (2,-1) {$(\partial G)^*$};
\node at (6.5,0) {\includegraphics[width=0.35\textwidth]{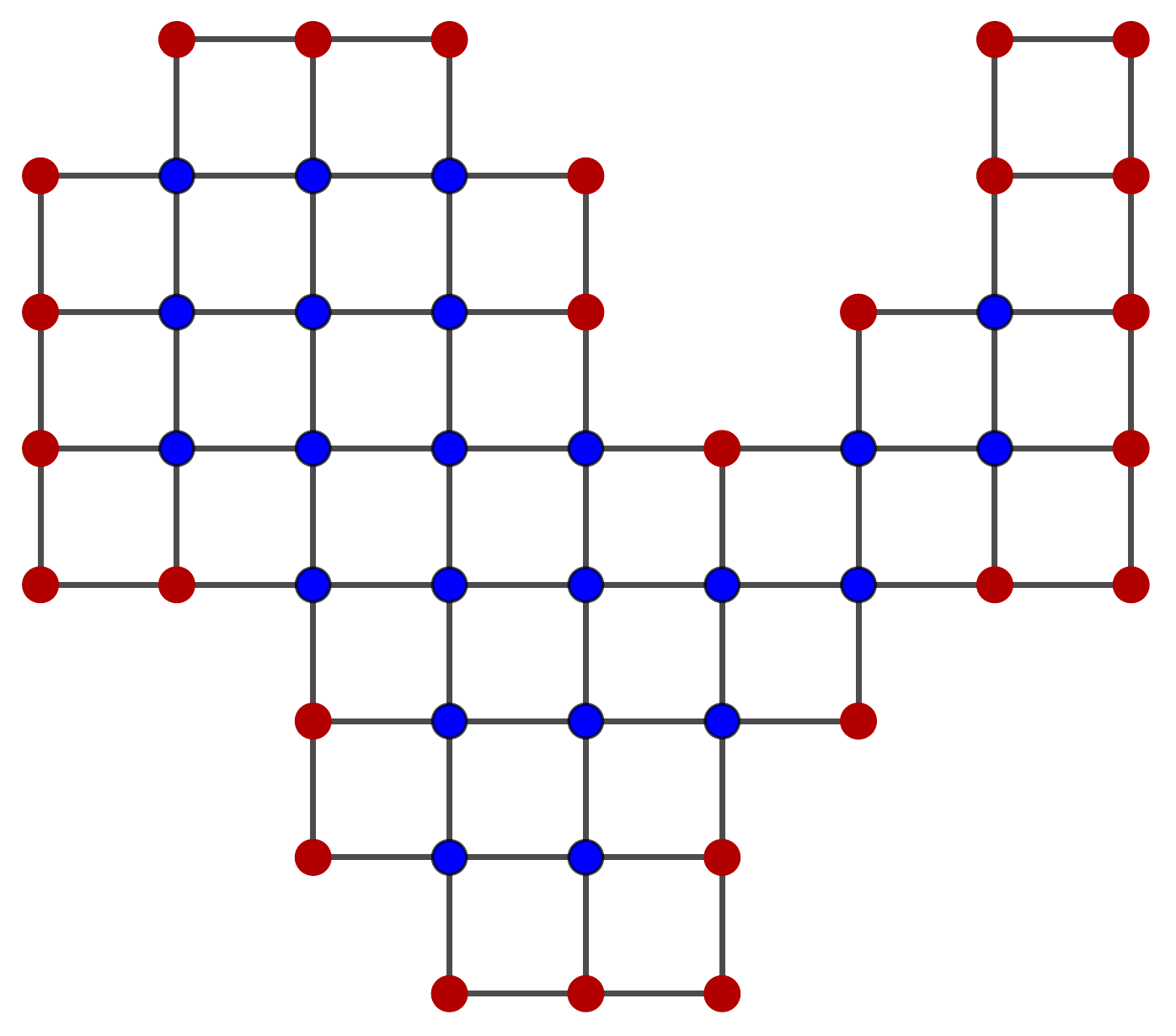}};
\node at (8.5,-1) {$\partial G$};
\end{tikzpicture}
\caption{Left: Chartrand-Erwin-Johns-Zhang boundary $(\partial G)^*$ (in red). Right: our notion of boundary $\partial G$ (in red).}
\end{figure}
\end{center}
\vspace{-15pt}
It is clear that vertices in the Chartrand-Erwin-Johns-Zhang boundary $(\partial G)^*$ are natural candidates to
be boundary vertices. However, when looking at examples, one might perhaps wonder if one should perhaps add some
of the other vertices as well. For this, we first consider the definition in Euclidean space. If $\Omega \subset \mathbb{R}^d$ is
a bounded domain with sufficiently nice boundary, then the Chartrand-Erwin-Johns-Zhang boundary will be a subset of $\partial \Omega$, the boundary of $\Omega$, but
it may well be a strict subset. The reason (see Fig. 4) is that in the presence of non-positively curved boundary, the boundary is actually being
classified as an interior point. Even in the presence of positively-curved boundary, the domain has to contain a point at a certain distance for a point
on the boundary to be classified as part of the Chartrand-Erwin-Johns-Zhang boundary.

\begin{center}
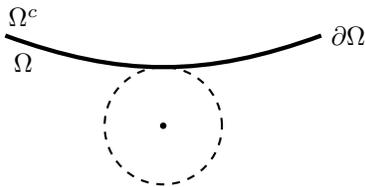
\begin{figure}[h!]
\begin{tikzpicture}[scale=1.2]
\draw [ultra thick] (0,0) to[out=340, in=200] (3.5,0);
\node at (0.2,0.2) {$\Omega^c$};
\node at (0.2, -0.3) {$\Omega$};
\filldraw (1.75, -1) circle (0.03cm);
\draw [thick, dashed] (1.75, -1) circle (0.65cm);
\node at (3.8, 0) {$\partial \Omega$};
\end{tikzpicture}
\caption{$(\partial G)^*$ in $\mathbb{R}^2$: in the case of non-positively curved boundary, the boundary does not satisfy the definition.}
\end{figure}
\end{center}
\vspace{-15pt}

Another motivating factor in our search for an alternative definition is that $(\partial G)^*$ does not satisfy an isoperimetric principle (see Fig. 5). An example is given by $n \times n$ grid graphs which nicely emulate the Euclidean domain $[0,1]^2 \subset \mathbb{R}^2$: however, independently of the number of vertices, $(\partial G)^*$ consists of exactly 4 corner vertices.
\vspace{0pt}
\begin{center}
\begin{figure}[h!]
\begin{tikzpicture}
\node at (0,0) {\includegraphics[width=0.14\textwidth]{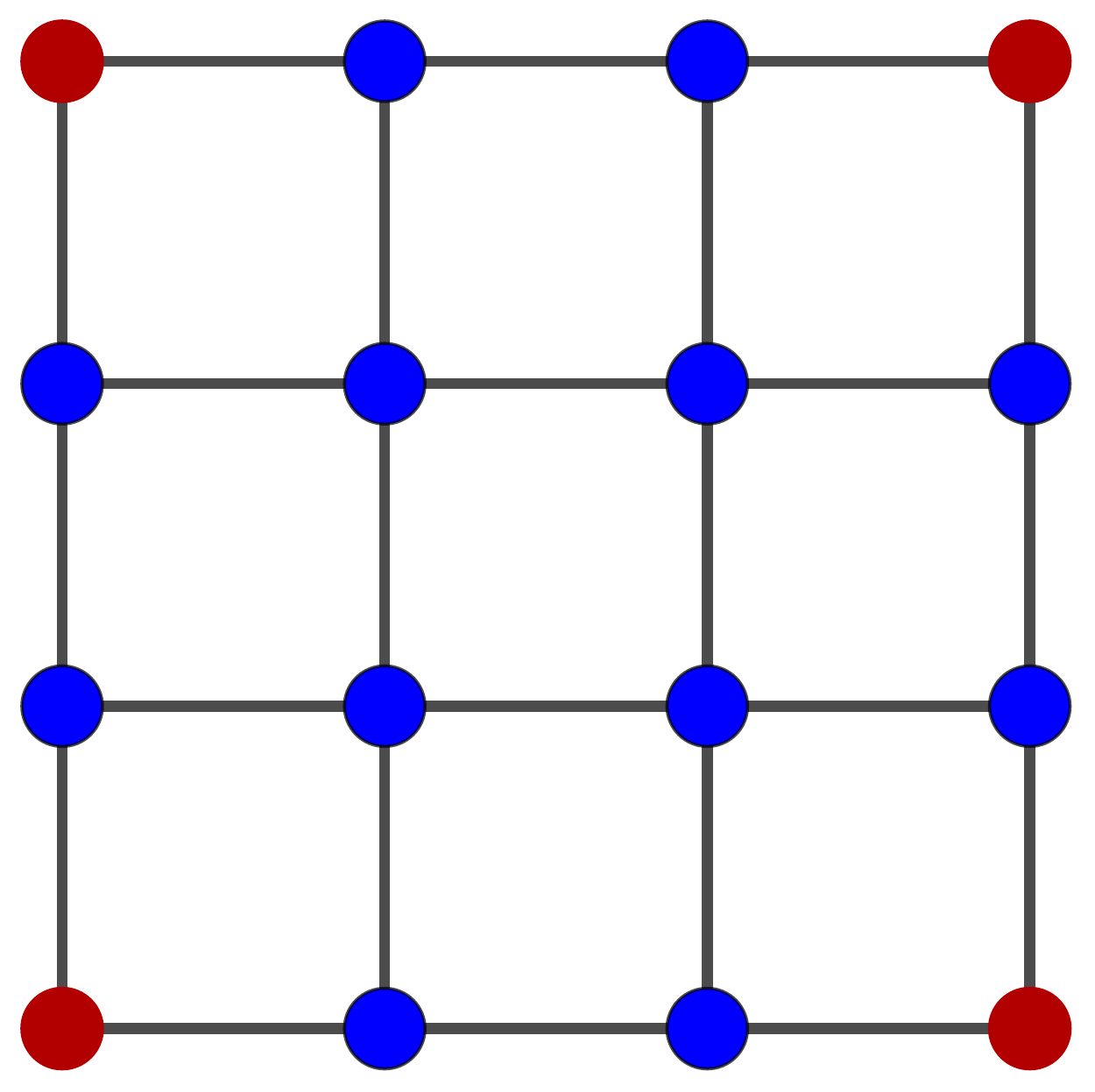}};
\node at (2.5,0) {\includegraphics[width=0.18\textwidth]{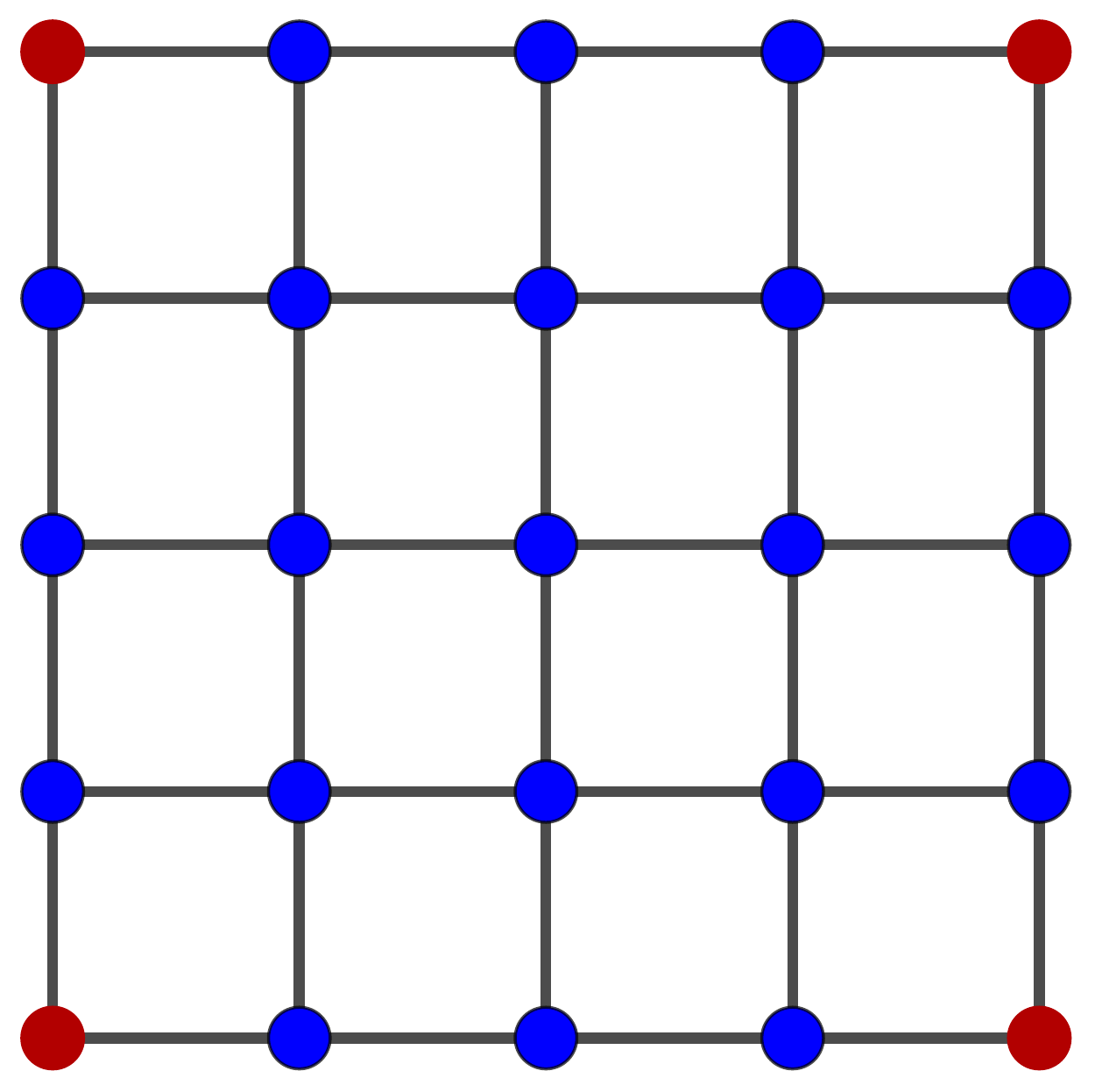}};
\node at (1.25,-1.4) {$(\partial G)^*$};
\node at (6,0) {\includegraphics[width=0.14\textwidth]{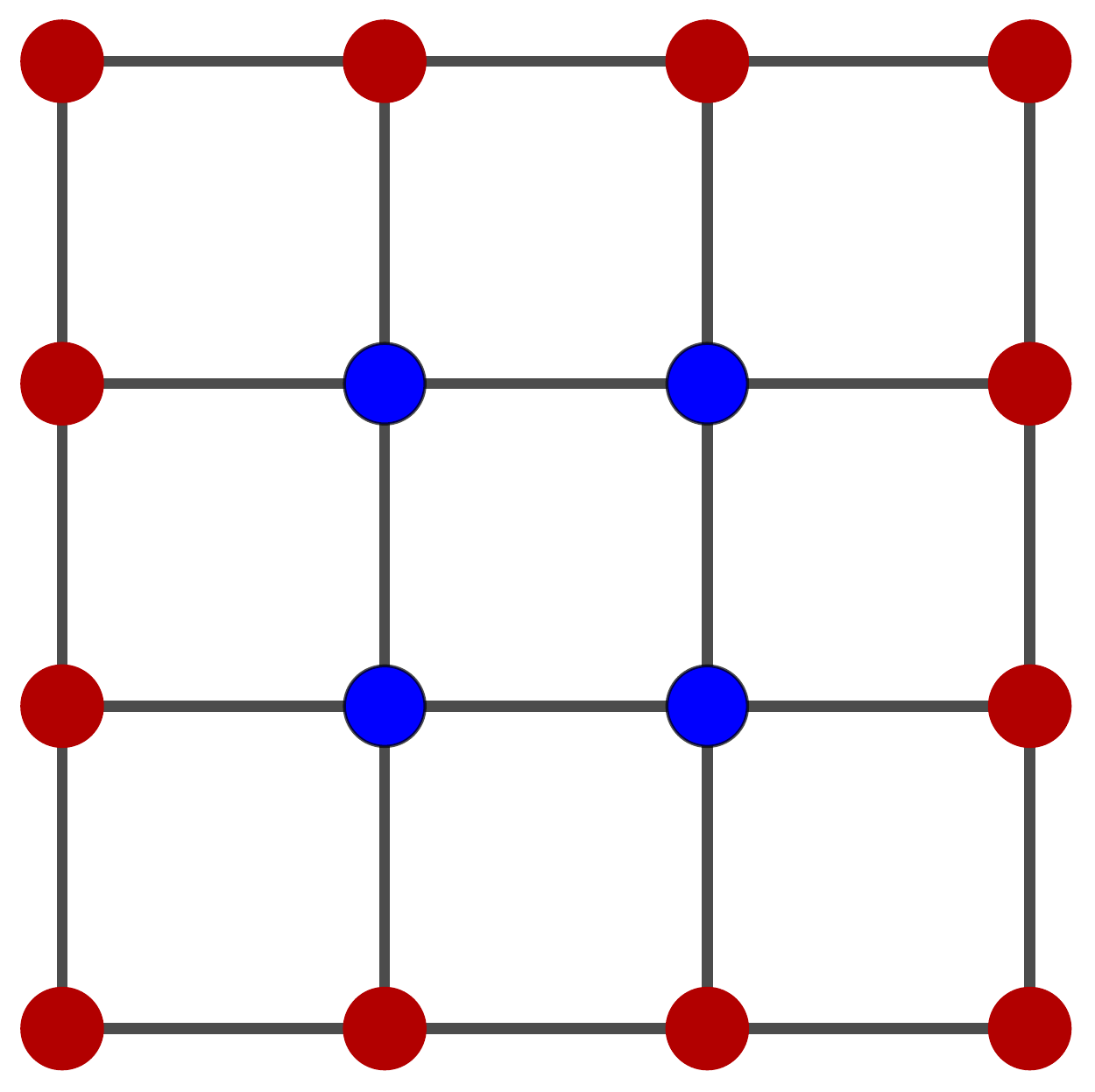}};
\node at (8.5,0) {\includegraphics[width=0.18\textwidth]{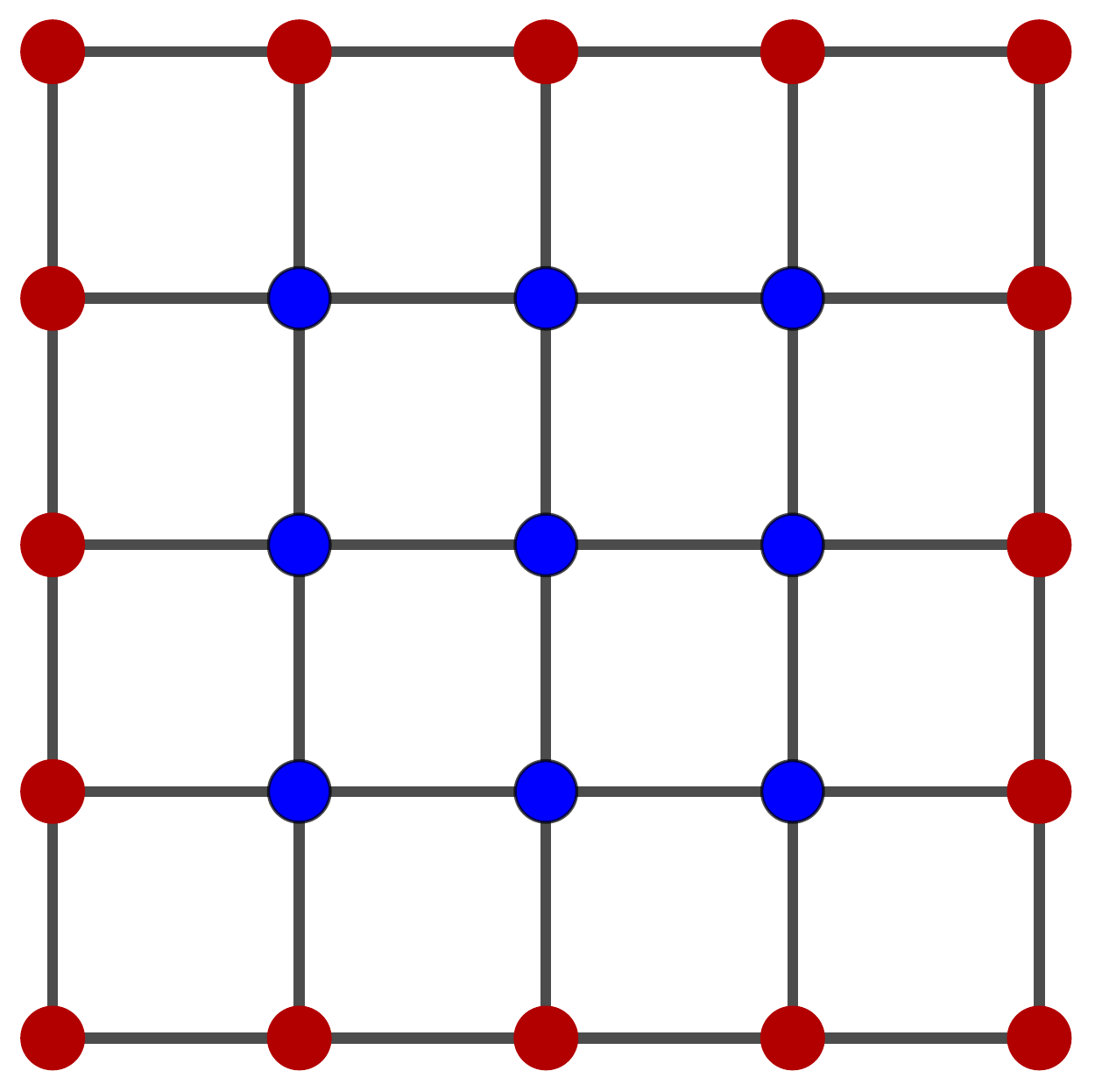}};
\node at (7.25,-1.4) {$\partial G$};
\end{tikzpicture}
\caption{Grid graphs satisfy $|(\partial G)^*| = 4$ and $ |\partial G| \sim 4|V|^{1/2}$.}
\end{figure}
\end{center}
\vspace{-5pt}

\textbf{A Notion of Boundary.} Recall that
$u \in (\partial G)^*$ means that we ask for the existence of another vertex $v$ such all the neighbors of $u$ are not further away from $v$ than $d(u,v)$. We relax the condition to ask for the existence of another vertex $v$ such that the \textit{average} neighbor of $u$ is
closer to $v$ than $d(u,v)$. Formally, 
$$\partial G = \left\{u \in V \big|  ~\exists v \in V:  ~  \frac{1}{\deg(u)} \sum_{(u, w) \in E} d(w,v) < d(u,v)  \right\}.$$
We first note that this is indeed a relaxation and that $\partial G \supseteq (\partial G)^*$.
\begin{proposition} We have $(\partial G)^* \subseteq \partial G$.
\end{proposition}
\begin{proof}
Suppose $u \in (\partial G)^*$. Then there exists $v \in V$ such that
$$ (u,w) \in E \implies d(v, w) \leq d(v,u).$$
Moreover, since there exists at least one shortest path from $v$ to $u$, we know that $u$ has at least one neighbor $z \in V$ such that $d(v,z) = d(v,u) - 1$. Therefore
\begin{align*}
  \frac{1}{\deg(u)} \sum_{(u, w) \in E} d(w,v) &\leq \frac{\deg(u) -1}{\deg(u)} d(u,v) + \frac{1}{\deg(u)} (d(u,v) -1) \\
  &< d(u,v).
  \end{align*}
\end{proof}

Another nice aspect of the definition is the connection to potential theory. Recall that for a graph on $n$ vertices, which we label $\left\{1,2,\dots, n\right\}$, the degree matrix $D \in \mathbb{R}^{n \times n}$ is a diagonal matrix where $D_{ii} = \deg(i)$ while the adjacency matrix $A \in \left\{0,1\right\}^{n \times n}$ is given by
$$ A_{ij} = \begin{cases} 1 \qquad &\mbox{if}~(i,j) \in E \\ 0 \qquad &\mbox{otherwise.} \end{cases}$$
We define a notion of a Laplacian matrix $L \in \mathbb{R}^{n \times n}$ via
$ L = D-A.$
With this definition, we can define our boundary in an alternative way: a vertex $u \in V$ is a boundary point if there exists another vertex $v \in V$ such that the distance function from $v$, 
$f_v(w) = d(w,v)$,
satisfies $(Lf_v)(u) > 0$. Formally, we can write
$$\partial G = \left\{u \in V \big|  ~\exists v \in V: \quad (L f_v)(u) > 0  \right\}.$$
Being able to write the definition in terms of the Laplacian $L$ opens up a connection to random walks which in turn allows for new techniques to be used. We note that the Laplacian of the distance function is a well-studied object in differential geometry, we refer to \cite{ca, ma} and references therein.

\begin{center}
\begin{figure}[h!]
\begin{tikzpicture}
\node at (0,0) {\includegraphics[width=0.12\textwidth]{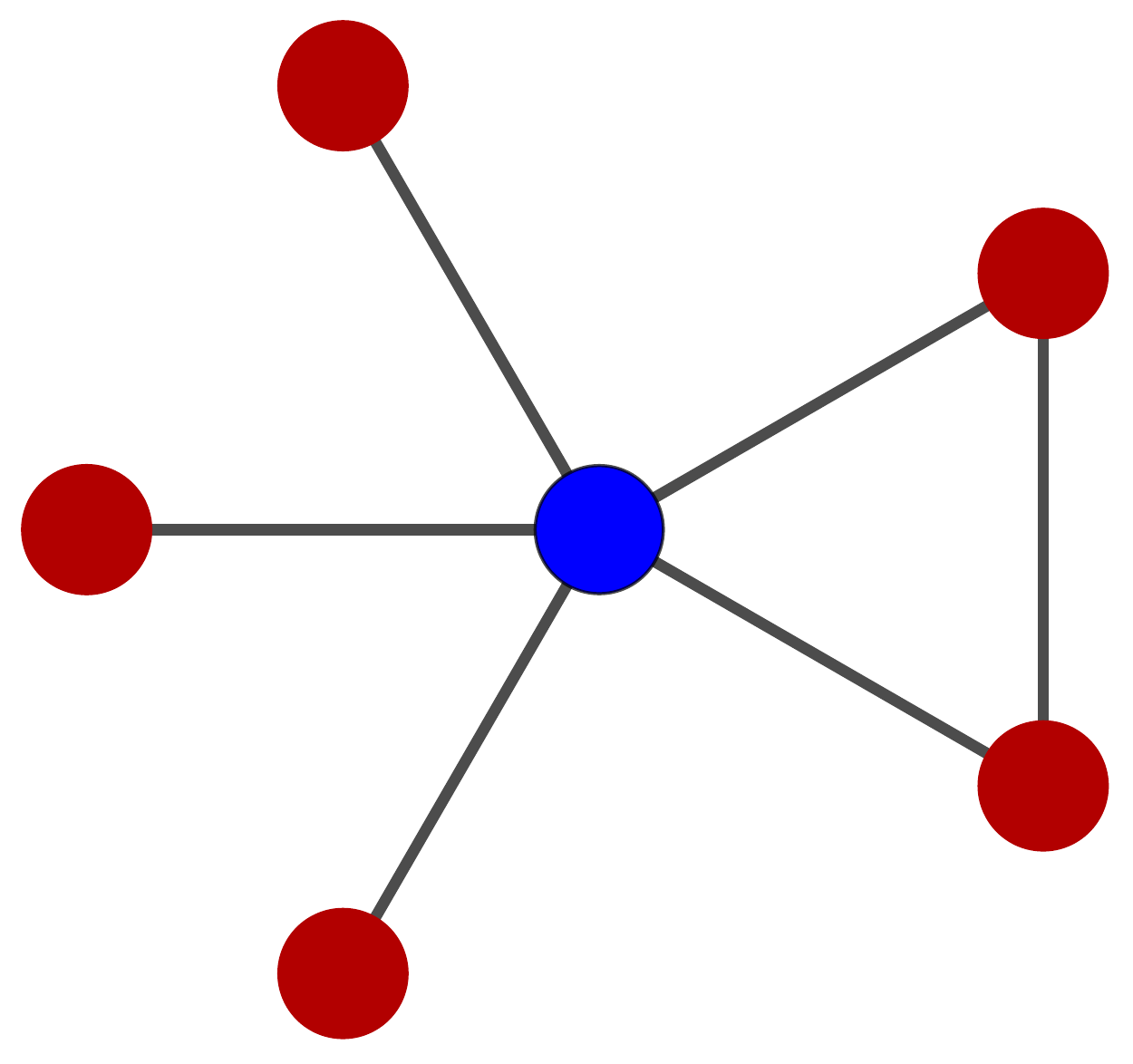}};
\node at (3,0) {\includegraphics[width=0.17\textwidth]{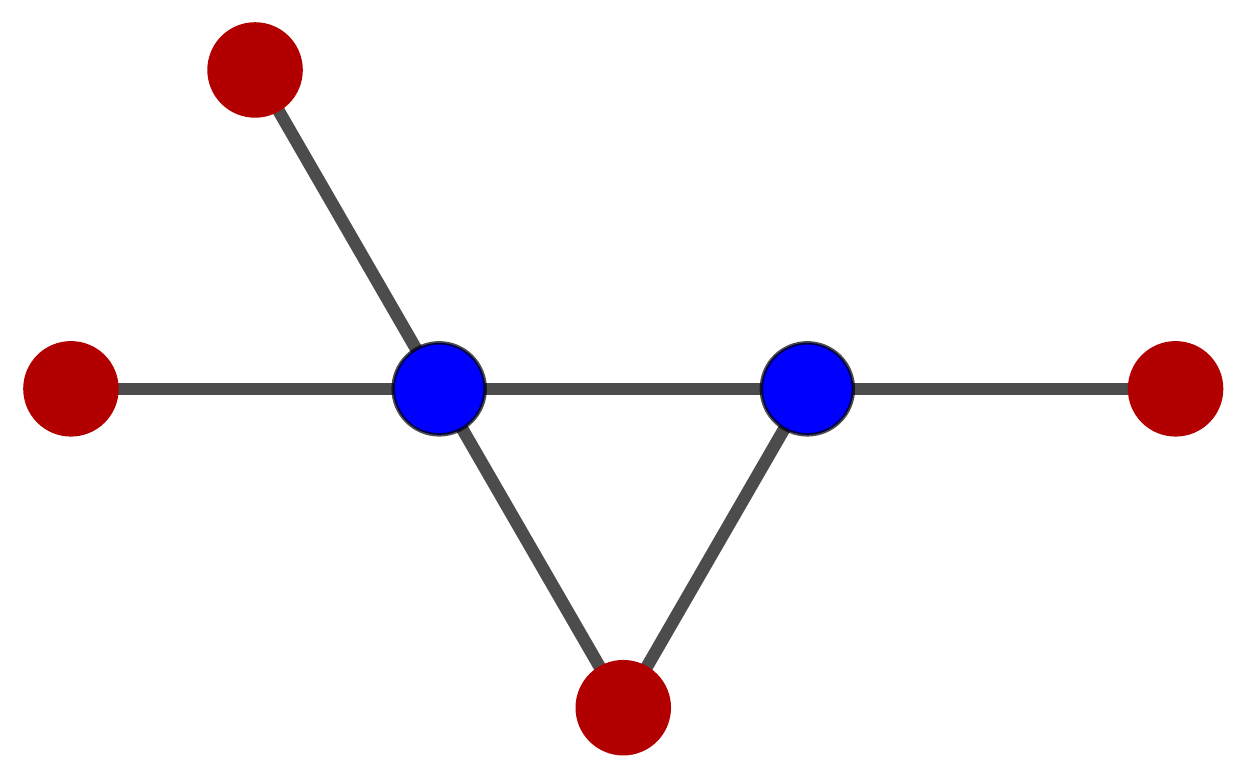}};
\node at (6,0) {\includegraphics[width=0.17\textwidth]{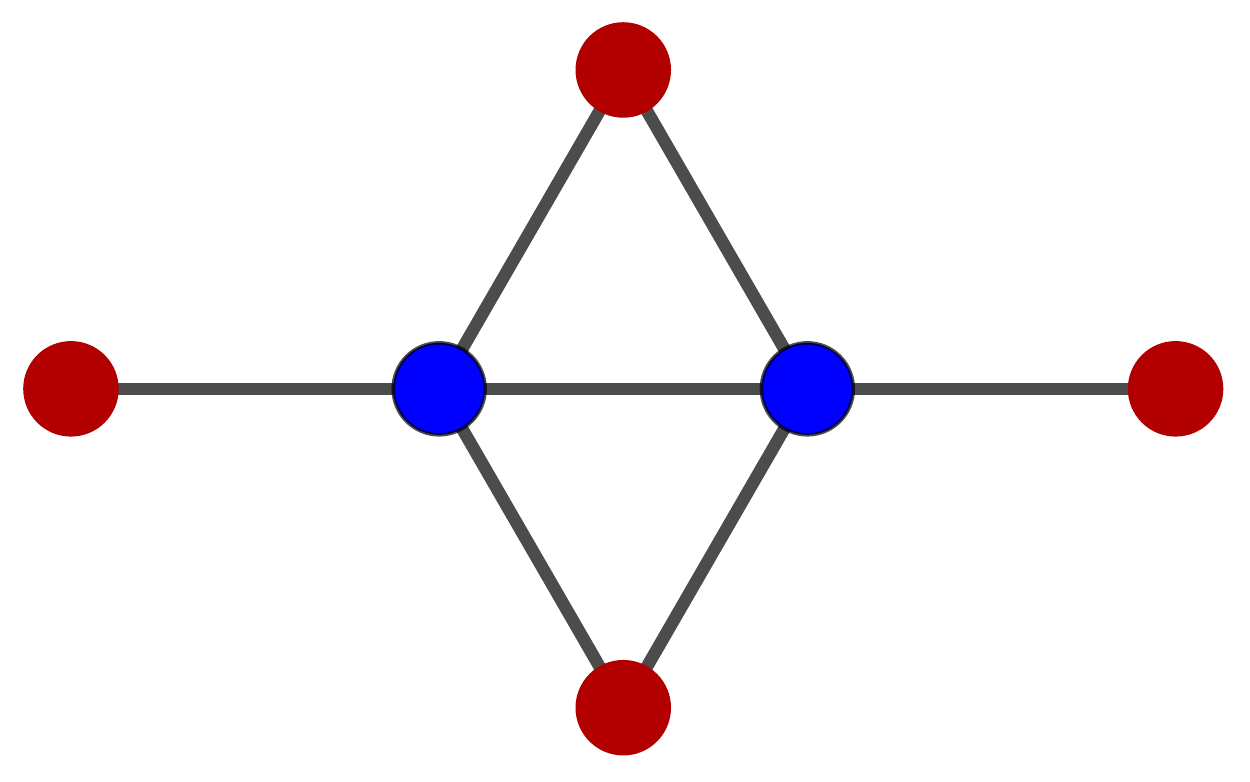}};
\node at (9,0) {\includegraphics[width=0.17\textwidth]{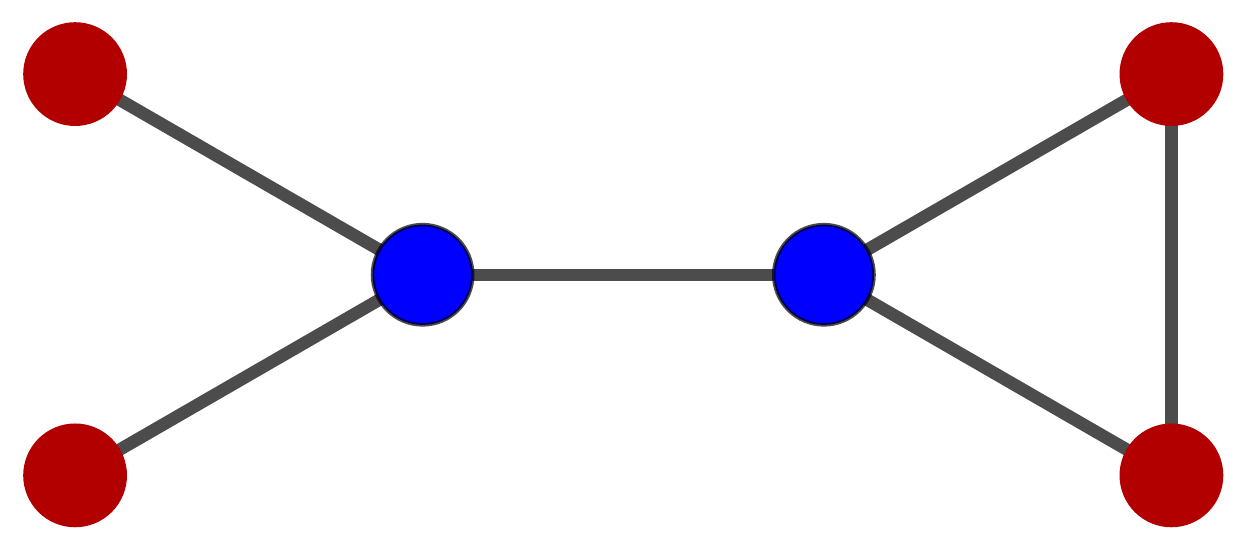}};
\node at (0,-1.7) {\includegraphics[width=0.12\textwidth]{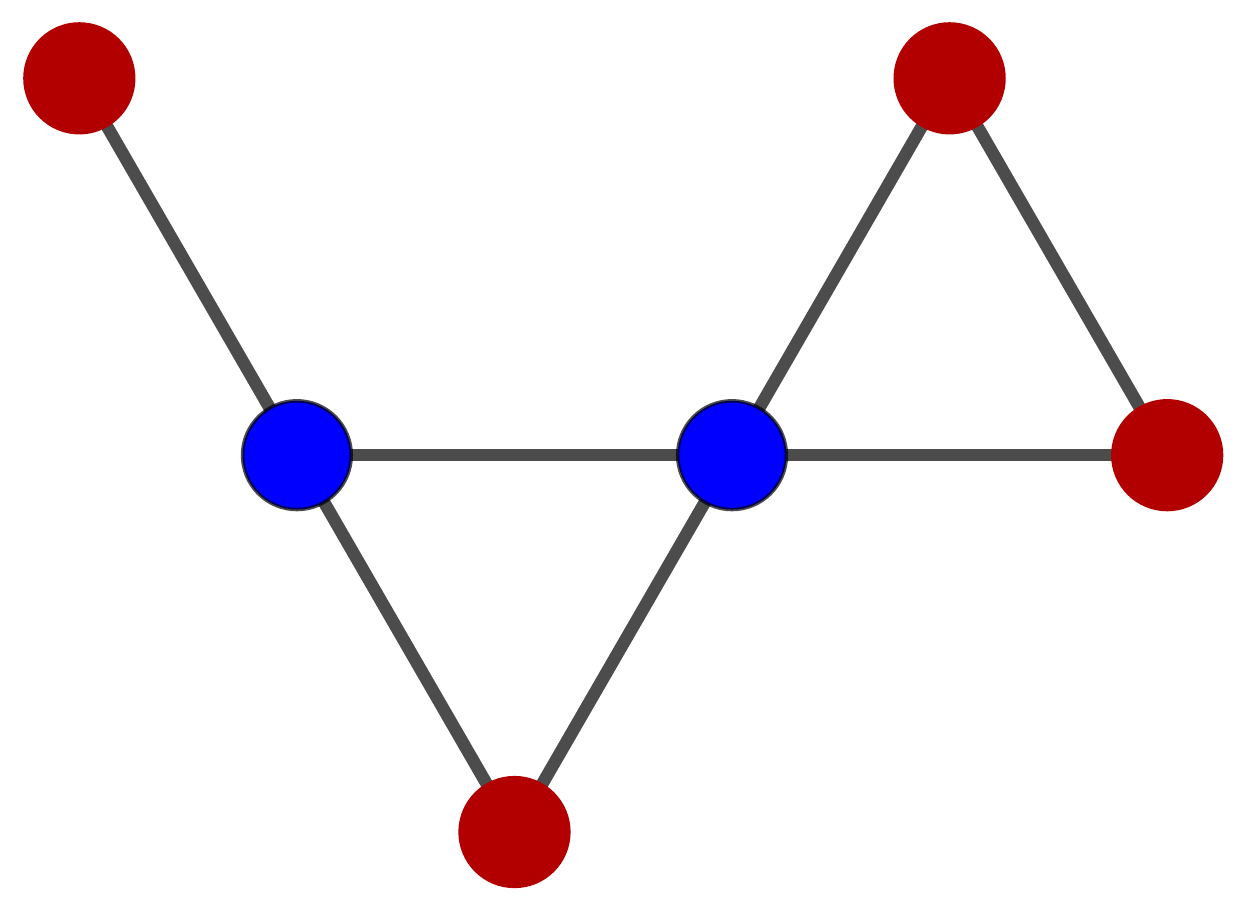}};
\node at (3,-1.7) {\includegraphics[width=0.12\textwidth]{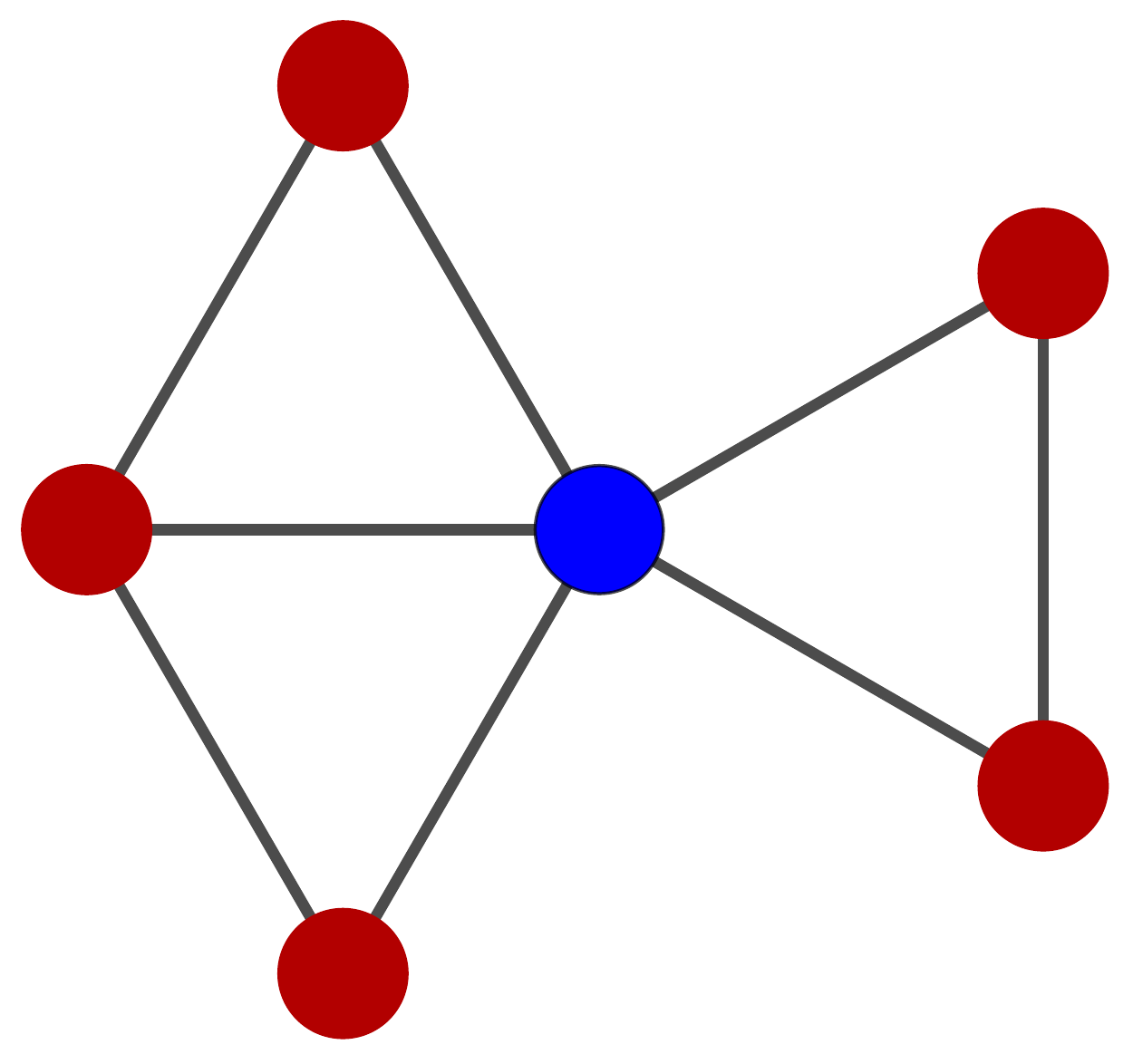}};
\node at (6,-1.7) {\includegraphics[width=0.11\textwidth]{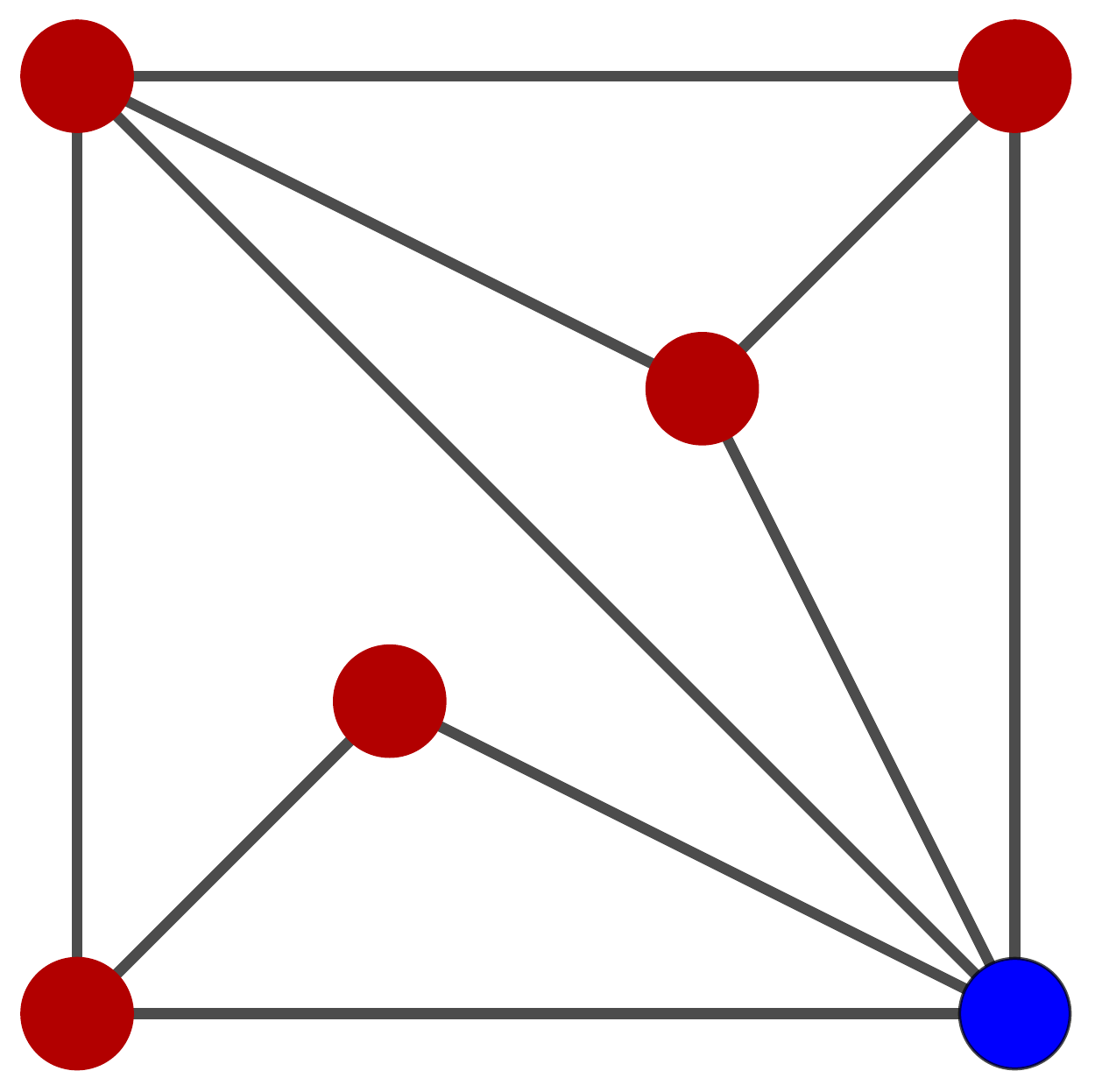}};
\node at (9,-1.7) {\includegraphics[width=0.11\textwidth]{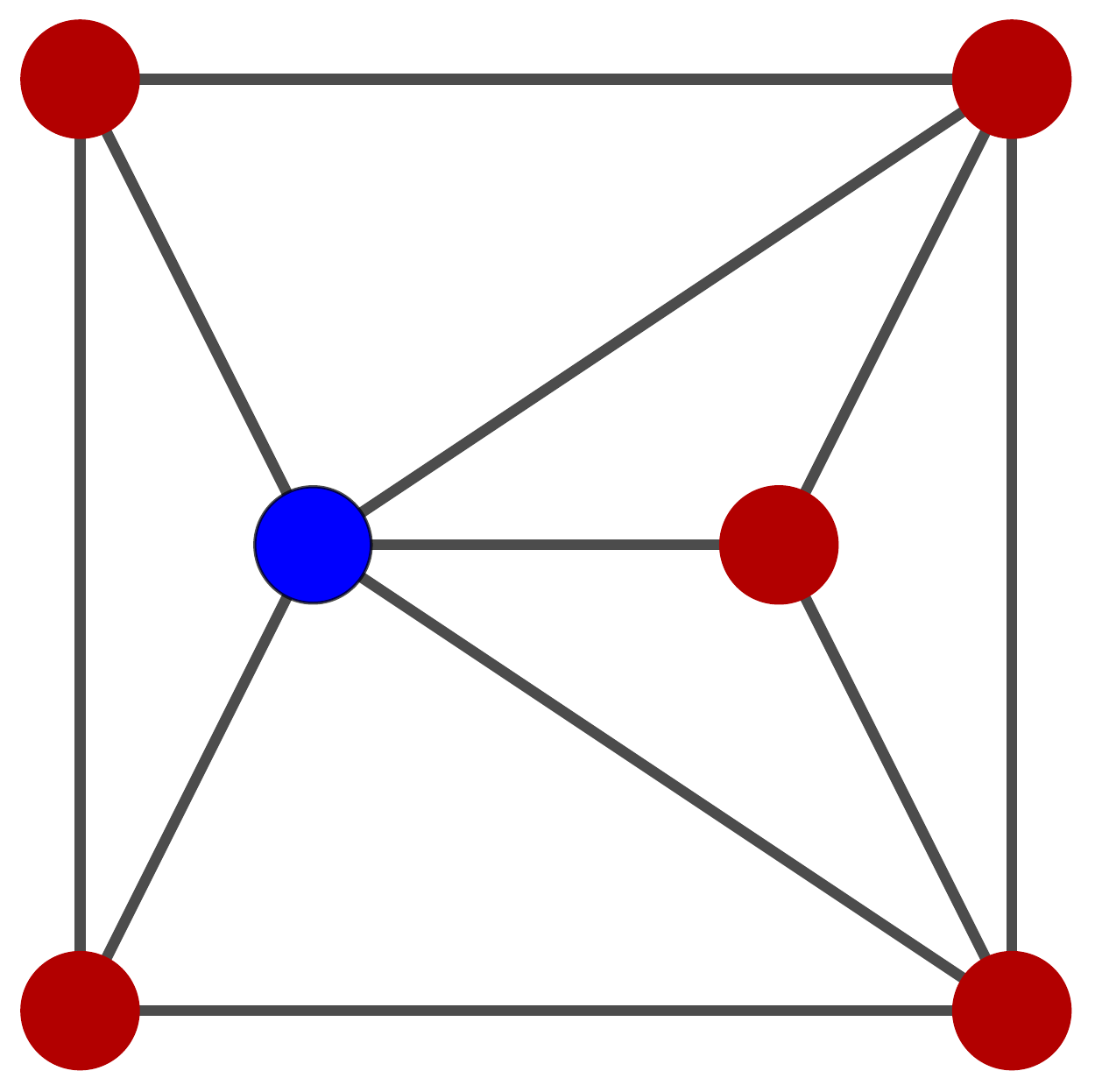}};
\end{tikzpicture}
\caption{Examples where $\partial G$ and $(\partial G)^*$ coincide.}
\end{figure}
\end{center}

\section{ Results}
\subsection{Basic Facts.} We start with basic results: on trees, $\partial G$ is what we expect.
\begin{proposition}[Trees and Leaves]
If $G$ is a tree, then $\partial G$ are the vertices of degree 1 (the `leaves'). For any connected graph, vertices of degree 1 are in $\partial G$. \end{proposition}

The second statement tells us that there always is a boundary: the proof is simple and exploits that any two vertices at distance $\diam(G)$ from each other are necessarily both in the boundary (this is true for both $(\partial G)^*$ and $\partial G$).

\begin{proposition}[Two boundary vertices]
For any connected graph with at least two vertices, we have $|\partial G| \geq 2$. If $|\partial G| = 2$, then $G$ is a path.
\end{proposition}
\begin{proof} It is easy to see that if $u,v \in V$ satisfy $d(u,v) = \diam(G)$, then $u,v \in \partial G$. Thus $|\partial G| \geq 2$. Suppose now that $|\partial G| = 2$: then, since $(\partial G)^* \subseteq \partial G$, we have $|(\partial G)^*| \leq 2$ and a result of  Hasegawa \& Saito \cite{has} implies that $G$ is a path. 
\end{proof}

Hasegawa \& Saito \cite{has} classify all graphs where $|(\partial G)^*| = 3$. This was later extended by M\"uller-P\'or-Sereni \cite{mu2} who classify all graphs for which $|(\partial G)^*| = 4$, there are nine different families of graphs. Given the complexity of characterizing $|(\partial G)^*| = 4$, it stands to reason that 
a characterization of graphs for which $|(\partial G)^*| = 5$ might be out of reach. Since $|\partial G| \geq |(\partial G)^*|$, the characterization problem for our boundary might be simpler; we have not pursued this here.

\subsection{Euclidean Approximation}
We continue with a simple result that shows that our notion of boundary makes sense for grid graphs approximating topologically simple Euclidean
domains with nice smooth boundary. Our setting will be as follows: we assume that $\Omega \subset \mathbb{R}^d$ is a bounded
domain and we assume that the graph $G$ is generated by taking a parameter $\lambda$, rescaling the
standard lattice $ \lambda \mathbb{Z}^d = \left\{\lambda v: v \in \mathbb{Z}^d\right\}$ and setting
$ V_{\lambda} = (\lambda \mathbb{Z}^d) \cap \Omega$
while the set of edges $E_{\lambda}$ is defined as $(u,v) \in E_{\lambda}$ whenever $\| u - v\|_2 = \lambda$. Note that each vertex
has exactly $2d$ neighbors unless it is close to the boundary of $\Omega$.
We will now prove that $\partial G_{}$ is comprised of vertices with degree less than $2d$ and vertices for which there exists a form of geodesic non-uniqueness (see also Fig. \ref{fig:geonon}).

\begin{proposition}[Domains in $\mathbb{R}^d$] If $u \in \partial G_{}$ has degree $2d$, then there exists a form of geodesic non-uniqueness:  there exists a vertex $v$ such that (1) either $d(u,v) = d(w,v)$ for a neighbor $w$ of $u$ or (2) there are two antipodal neighbors $v \pm \lambda e_i$ with
$$ d(u \pm \lambda e_i , v) = d(u, v) - 1.$$
 \end{proposition}
 
Case (1) can be seen for a cycle graph with odd cardinality. Case (2) is illustrated in Fig. \ref{fig:geonon}. The geodesic path from $v$ to $u$ is not unique (something that is very common in a grid graph) and that moreover there are two different shortest paths starting in \textit{opposite} directions. Both cases indicate that $\Omega$ has nontrivial topology.

\begin{center}
\begin{figure}[h!]
\begin{tikzpicture}
\node at (0,0) {\includegraphics[width=0.25\textwidth]{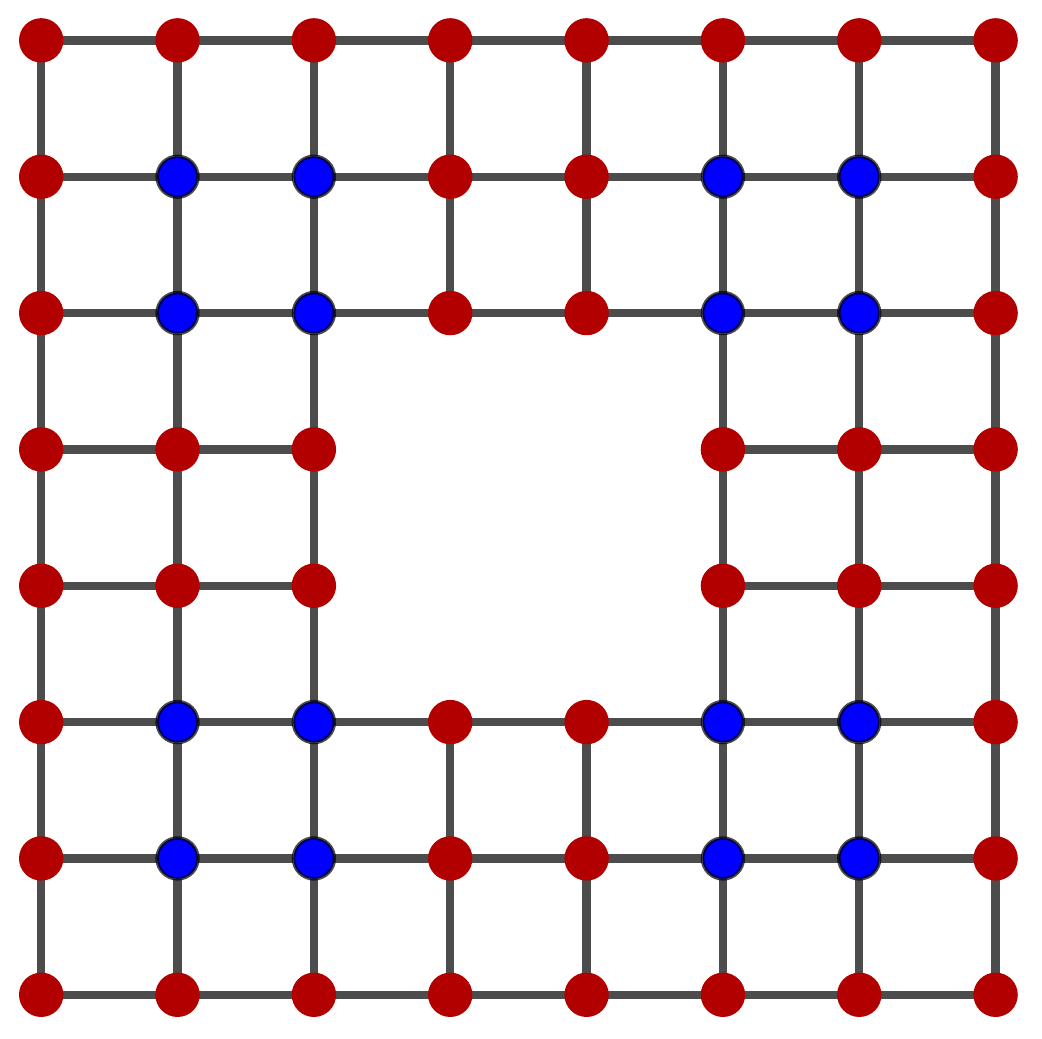}};
\node at (5,0) {\includegraphics[width=0.25\textwidth]{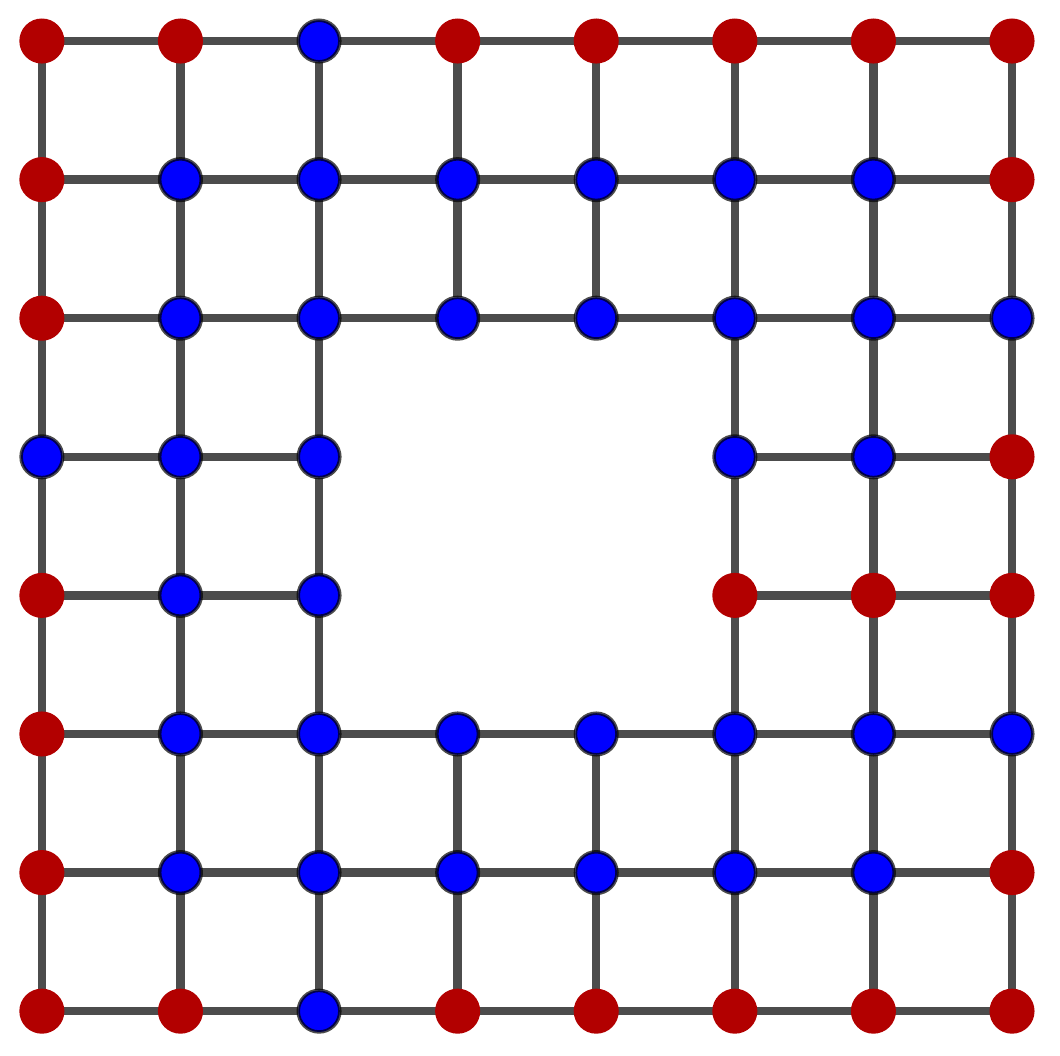}};
\draw [ultra thick] (4.37, 0.21) circle (0.2cm);
\node at (4.75, 0.2) {$v$};
\draw [ultra thick] (5.6, -0.2) circle (0.2cm);
\node at (5.2, -0.25) {$u$};
\end{tikzpicture}
\caption{Left: $\partial G$ on a graph. Right: points identified as being in $\partial G$ when judging from the vertex $v$. Non-uniqueness of shortest paths leads to additional structure.}
\label{fig:geonon}
\end{figure}
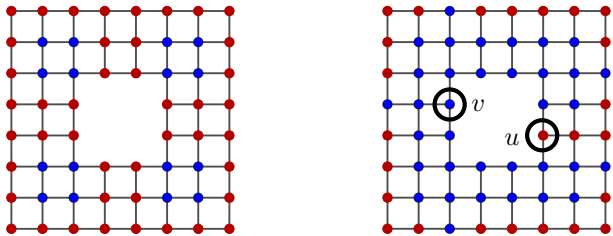
\end{center}

 \vspace{-20pt}

\subsection{Isoperimetric Inequality.} We can now discuss our main result: our definition of boundary $\partial G$ satisfies an isoperimetric principle. The isoperimetric inequality says that, when defined, the surface area of a domain $\Omega \subset \mathbb{R}^d$ satisfies
$$ \mbox{area}(\partial \Omega) \geq c\cdot \mbox{vol}(\Omega)^{\frac{d-1}{d}},$$
 where $\mbox{area}$ denotes the $(d-1)$-dimensional surface area: this inequality is true for all `reasonable' definitions of area and volume (of which there are several) and the optimal constant $c$ is attained when $\Omega$ is a ball. The most direct analogue of the isoperimetric inequality, an inequality relating $|\partial G|$ and the cardinality of the vertices $|V|$ cannot be true: 
 we face an additional obstruction that does not exist in Euclidean space and that is illustrated in Fig. \ref{fig:ob}. Not only do paths always have $|\partial G| = 2$ independently of their length, it is always possible to attach long paths to any given graph which ensures one can add as many non-boundary vertices as one wishes at the cost of adding at most two additional boundary vertices. 
   
\begin{center}
\begin{figure}[h!]
\begin{tikzpicture}
\node at (0,0) {\includegraphics[width=0.18\textwidth]{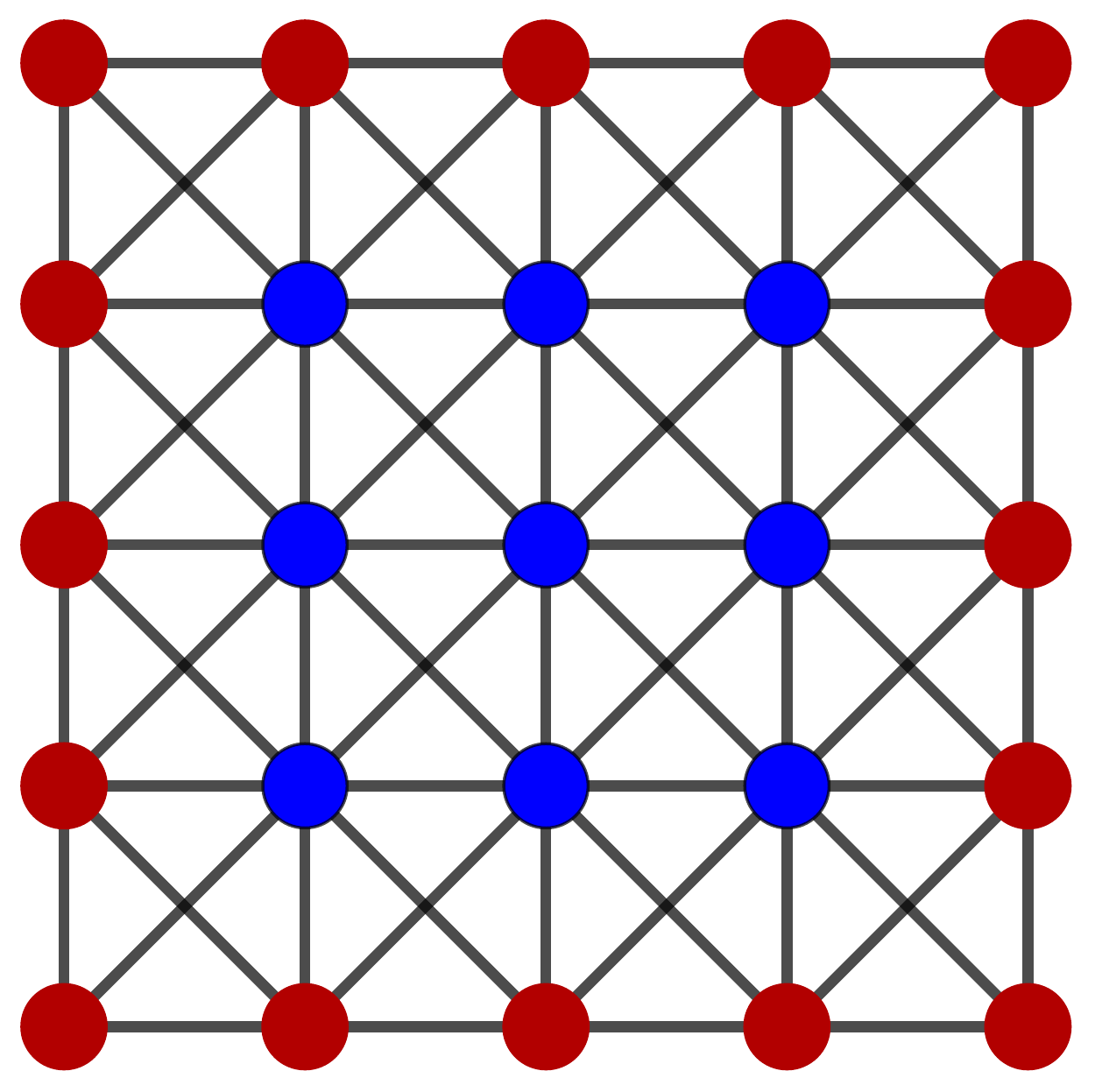}};
\node at (5,0) {\includegraphics[width=0.38\textwidth]{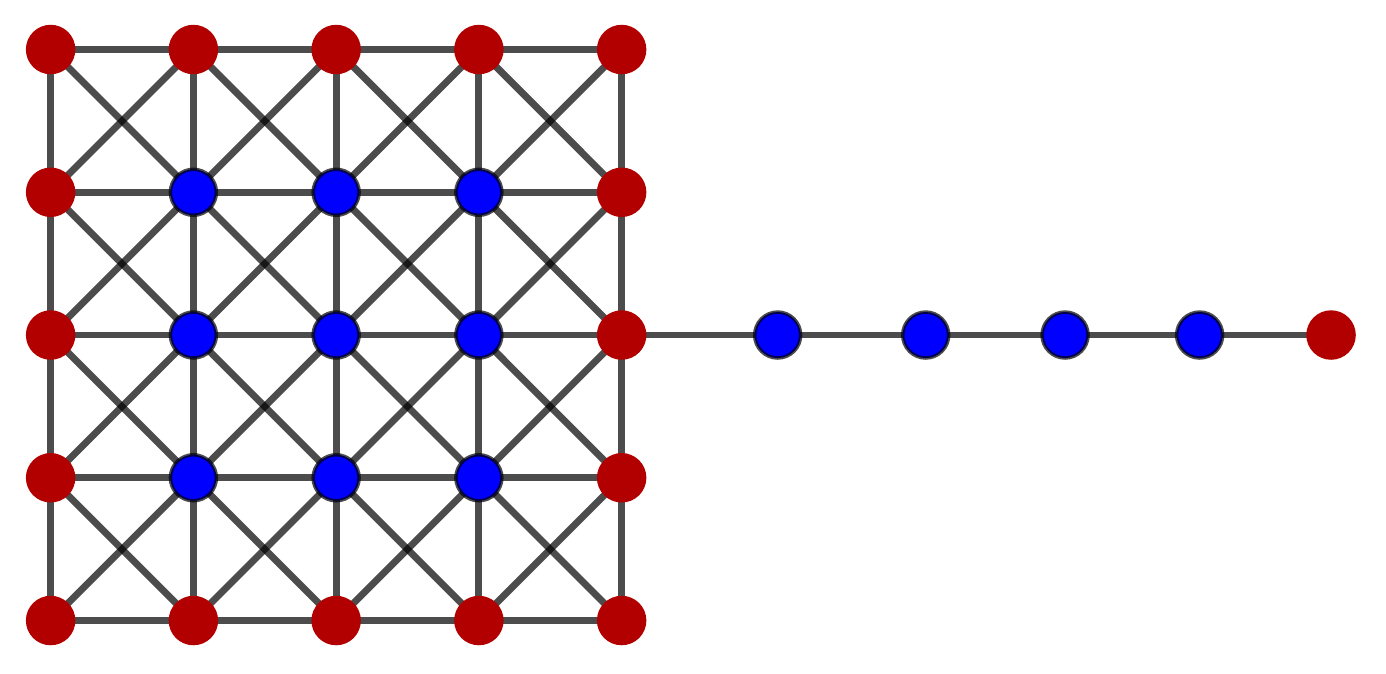}};
\end{tikzpicture}
\caption{It is always possible to create a lot of non-boundary edges by adding paths to any given graph.}
\label{fig:ob}
\end{figure}
\end{center}
\vspace{-10pt}

However, such long paths will also, once they are very long, eventually increase the diameter of the graph; taking this into consideration and adding the diameter as a relevant term, an isoperimetric inequality becomes possible. We first note that such an inequality exists in $\mathbb{R}^n$: the isodiametric inequality implies that 
$$ \diam(\Omega)  \geq c_2 \cdot \mbox{vol}(\Omega)^{1/d}.$$
 This implies, when combined with the classical isoperimetric inequality, that
$$ \mbox{area}(\partial \Omega) \geq c\cdot \mbox{vol}(\Omega)^{\frac{d-1}{d}}  \geq c \cdot c_2 \cdot \frac{\mbox{vol}(\Omega)}{\mbox{diam}(\Omega)}.$$
Since both inequalities are attained for the ball, the best constant is attained when $\Omega$ is a ball. Our main result will show that the isoperimetric principle, when stated in this particular form, has an analogue on graphs.

\begin{theorem}[Isoperimetric Inequality] If $G$ has maximal degree $\Delta$, then
$$ | \partial G| \geq \frac{1}{2\Delta} \frac{|V|}{\diam(G)}.$$
\end{theorem}

It is easy to see that the inequality has the correct scaling in $|V|$ and $\diam(G)$. Consider $\Omega = [0,1]^d$ and approximate it with a grid graph $G$. We have 
$ \Delta = 2d$, while, as $|V| \rightarrow \infty$, we have $\diam(G) \sim d \cdot |V|^{(d-1)/d}$ and
$ | \partial G| \sim 2d \cdot |V|^{(d-1)/d}$. It is an interesting question whether the inequality has the sharp scaling in $\Delta$. This might not be the case: when $\Delta \sim |V|$, then our inequality does not imply any nontrivial bound while a result of  M\"uller-P\'or-Sereni \cite{mu} implies
$$ |\partial G| \geq |(\partial G)^*| \geq \log_2(\Delta + 2).$$
 The inequality $|(\partial G)^*| \geq \log_2(\Delta + 2)$ is known to be sharp up to constants, see \cite{mu}. It is certainly conceivable that $ |\partial G|  \geq \log_2(\Delta + 2)$ might not be.

\subsection{A Refined Isoperimetric Inequality.} The proof of Theorem 1 gives more information. The definition of $\partial G$ is as follows: for each vertex $v$, we add any vertex $u$ with the property that a randomly chosen neighbor of $u$ is, in expectation, closer to $v$ than $d(u,v)$. In particular, by looking at a single vertex $v$ we can already identify parts of the boundary ($\partial G$ is given by the union of these sets over all $v$). What is particularly interesting is that the proof of Theorem 1 implies a \textit{much} stronger result:  each individual vertex $v$ already creates a large boundary.

\begin{theorem} If $G$ is a connected graph with maximal degree $\Delta$, then for all $v \in V$
$$ \left| \left\{u \in V \big|  \quad  \frac{1}{\deg(u)} \sum_{(u, w) \in E} d(w,v) < d(u,v)  \right\}\right| \geq \frac{1}{2\Delta} \frac{|V|}{\diam(G)}.$$
\end{theorem}
An example is seen in Fig. \ref{fig:imp}: we clearly only recover a subset of the
boundary but Theorem 2 guarantees that this subset is not too small. We are not aware of this principle having any analogue in the continuous setting (which inspired the discussion in \S 2.5 where a continuous analogue is proposed).
\vspace{-5pt}
\begin{center}
\begin{figure}[h!]
\begin{tikzpicture}[scale=1]
\node at (0,0) {\includegraphics[width=0.6\textwidth]{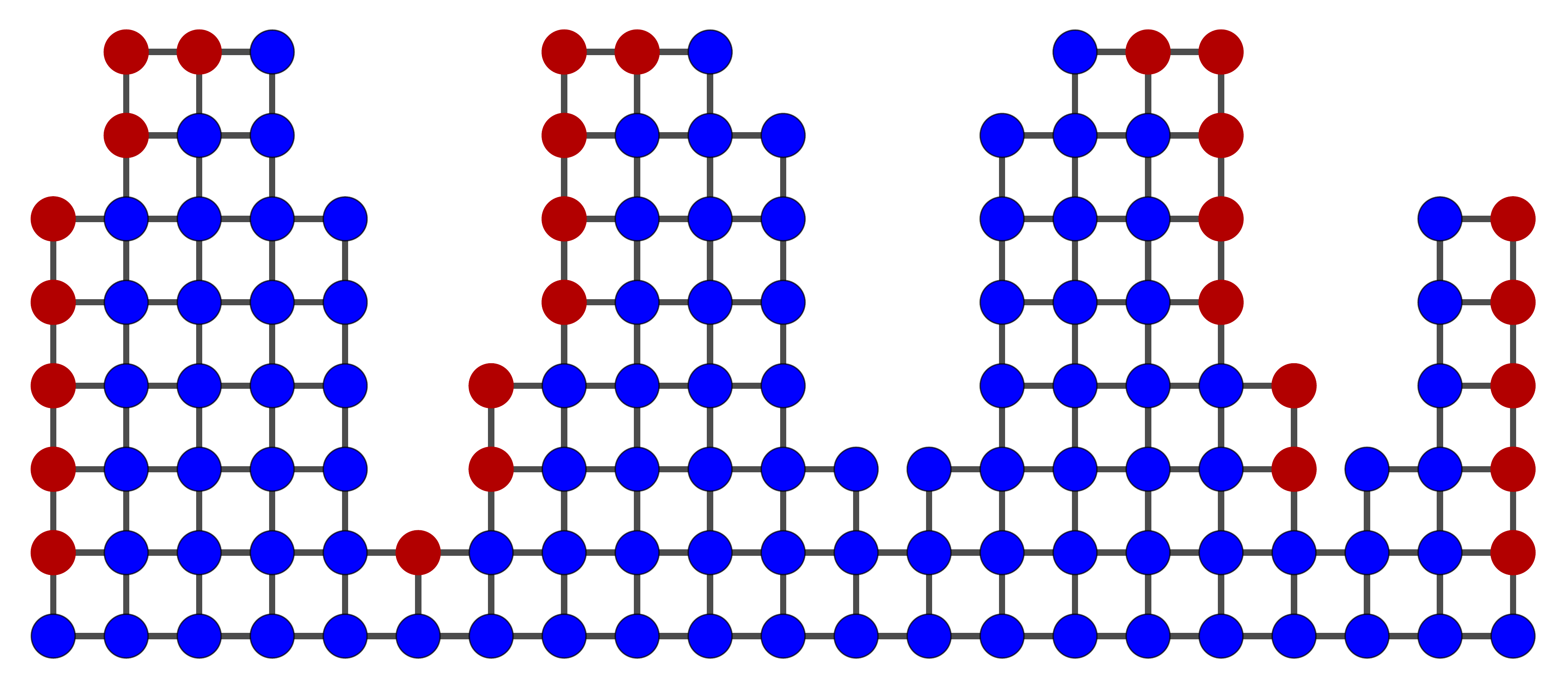}};
\draw [ultra thick] (0.71, -1.41) circle (0.2cm);
\node at (0.95, -1.7) {$v$};
\end{tikzpicture}
\vspace{-10pt}
\caption{The part of $\partial G$ identifiable from $v$.}
\label{fig:imp}
\end{figure}
\end{center}
\vspace{-15pt}

The proof of Theorem 2 seems to suggest that this bound might be somewhat close to sharp:  tracking the precise constants, we get  the slightly better bound 
$$ \left| \left\{u \in V \big|  \quad  \frac{1}{\deg(u)} \sum_{(u, w) \in E} d(w,v) < d(u,v)  \right\}\right| \geq \frac{|V|-1}{2 \Delta (\diam(G)-1) + 1}.$$
However, Theorem 2 being close to sharp for a single vertex $v$ need not necessarily imply that Theorem 1 is sharp since $\partial G$ is defined as the union
over all vertices that are seen as being boundary by \textit{some} other vertex. This raises the question of whether it is possible for Theorem 2 to be close to sharp
for all vertices simultaneously and, if so, whether the parts of the boundary identified by these different vertices are different or largely overlapping.

\subsection{Back to Euclidean space.} Theorem 2 (see also Fig. \ref{fig:imp}) seems to suggest a new type of isoperimetric principle 
in Euclidean space. 
Let $\Omega \subset \mathbb{R}^d$ be a bounded domain with smooth boundary and let $x \in \Omega$ be an arbitrary point. We define a subset
$(\partial \Omega)_x \subseteq \partial \Omega$ via
$$ (\partial \Omega)_x = \left\{y \in \partial \Omega: \mbox{the geodesic from}~x~\mbox{to}~y~\mbox{arrives non-tangentially} \right\}.$$
We note that the geodesic is defined as the shortest path $\gamma:[0,1] \rightarrow \Omega$ with $\gamma(0) = x$ and $\gamma(1) = y$. We say that it arrives non-tangentially if $\left\langle \gamma'(1), \nu \right\rangle \neq 0$, where $\nu$ is the normal vector of $\partial \Omega$ in $y$.
This could be understood, in some sense, as the continuous analogue of our definition of graph boundary. If $\Omega$ is convex and $x$ is not on the boundary, then $(\partial \Omega)_x = \partial \Omega$
and not much changes. For non-convex domains $\Omega$, we observe that $(\partial \Omega)_x$ can be much smaller than $\partial \Omega$. 
\begin{center}
\begin{figure}[h!]
\begin{tikzpicture}[scale=0.8]
\draw [ultra thick] (0,0) ellipse (1cm and 1.4cm);
\filldraw (0, 0.2) circle (0.07cm);
\node at (0.2, -0.2) {$x$};
\draw [] (0, 0.2) -- (0, 1.4);
\draw [] (0, 0.2) -- (1, 0);
\draw [] (0, 0.2) -- (-1, -0.3);
\draw [] (0, 0.2) -- (-0.5, -1.2);
\draw[ultra thick] (3,0) to[out=90, in =180] (4,1.4) to[out=0, in=180] (5, 0.4);
\draw[ultra thick] (3,0) to[out=270, in =180] (4,-1.4) to[out=0, in=180] (5, -0.4);
\draw[] (5, 0.4) -- (5.5, 0.4);
\draw[] (5.5, 0.4) -- (6.5, 0.4);
\draw[] (6.5, 0.4) -- (7, 0.4);
\draw[] (5, -0.4) -- (5.5, -0.4);
\draw[] (5.5, -0.4) -- (6.5, -0.4);
\draw[] (6.5, -0.4) -- (7, -0.4);
\draw[ultra thick] (7, -0.4) -- (7, 0.4);
\draw [ultra thick] (4.25, -0.75) -- (4.25, 0.75);
\filldraw (3.5, 0) circle (0.05cm);
\node at (3.5, -0.25) {$x$};
\draw [] (3.5,0) -- (4.25, 0.75) -- (4.8, 0.45);
\draw (12,-0.24) -- (9,0) -- (12,0.24);
\filldraw (9,0) circle (0.07cm);
\draw [ultra thick,domain=0:5] plot ({9 + 3*cos(\x)}, {3*sin(\x)});
\draw [ultra thick,domain=355:360] plot ({9 + 3*cos(\x)}, {3*sin(\x)});
\draw [ultra thick,domain=0:5] plot ({9 + 2.98*cos(\x)}, {2.98*sin(\x)});
\draw [ultra thick,domain=355:360] plot ({9 + 2.98*cos(\x)}, {2.98*sin(\x)});
\node at (9, -0.3) {$x$};
\end{tikzpicture}
\vspace{-10pt}
\caption{Various examples of $(\partial \Omega)_x$.}
\label{fig:2D}
\end{figure}
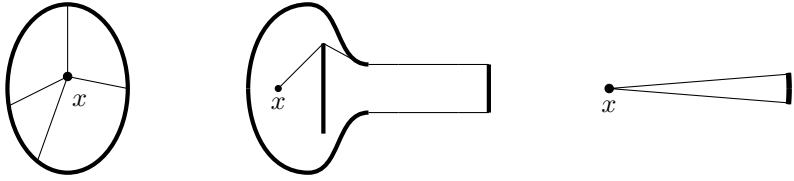
\end{center}
\vspace{-10pt}
Examples are shown in Fig \ref{fig:2D}: the domain in the middle is comprised of a ball-like domain (with a slit removed) to which a large rectangular domain has been added. $(\partial \Omega)_x$ contains all the domain of the ball-region but does not contain the long axes of the rectangle (which can be arbitrarily long) because the geodesic in that region arrives tangentially. The example on the right shows an example where $(\partial \Omega)_x$ is concentrated on a small part of the curved part of the boundary $\partial \Omega$.

\begin{proposition}
Let $\Omega \subset \mathbb{R}^d$ be a bounded domain with smooth boundary and $x \in \Omega$. If the distance function from $x$ is sufficiently smooth and satisfies, for all $y \neq x$,
$$ \Delta_y d(x, y) \geq \frac{d-1}{d(x, y)},$$
then
$$ \forall x \in \Omega \qquad  | (\partial \Omega)_x| \geq (d-1) \frac{|\Omega|}{\diam(\Omega)}.$$
\end{proposition}
It seems like an interesting problem to understand for which domains $\Omega \subset \mathbb{R}^d$ such an inequality is true, how the optimal constant behaves depending on the dimension and what extremal domains could look like. It is easy to see that for $d=2$, the constant is optimal up to at most a factor of 2 (see Fig. \ref{fig:2D}, right): let $\Omega$ to be a narrow disk-sector with radius 1 and opening angle $0 \leq \alpha \ll 1$ in $\mathbb{R}^2$ and choose $x$ to be in the corner. $\Omega_x$ is then exactly the curved part of the boundary and
$$| (\partial \Omega)_x| = 2 \pi r \alpha \qquad |\Omega| = r^2 \pi \alpha \quad \mbox{and} \quad \mbox{diam}(\Omega) = r.$$
 The condition
$ \Delta d(x, \cdot) \geq (d-1)/d(x, \cdot)$
is always satisfied in convex domains: note that 
$$ \Delta \sqrt{x_1^2 + \dots + x_d^2} = \frac{d-1}{\sqrt{x_1^2 + \dots + x_d^2}}$$
is an exact equation in $\mathbb{R}^d$. In the case of a polygonal boundary, we can use the same identity to argue that if a geodesic is sliding along
part of the boundary, then the Laplacian will be determined by the distance between $y$ and the boundary (along the geodesic) and the condition is satisfied. 
We expect $ \Delta d(x, \cdot) \geq (d-1)/d(x, \cdot)$ to be violated when there is non-uniqueness of geodesics: is the right condition for $\Omega \subset \mathbb{R}^2$ to have such an inequality perhaps that $\Omega$ is simply connected?

\section{Proofs}

\subsection{Proof of Proposition 2}
\begin{proof} 
Let $G$ be a connected graph with at least two vertices. Let $u$ be a vertex of degree 1 and let $v$ be any other vertex. Then the path from $v$ to $u$ is unique (because $G$ is a tree) and the path terminates in $u$. This means $u$ has exactly one neighbor and that this neighbor is closer to $v$ than $u$ and thus
$$ \frac{1}{\mbox{deg}(u)} \sum_{(u, w) \in E} d(w,v) = d(u,v) - 1 < d(u,v) $$
implying that $u \in \partial G$.  Suppose conversely now that
$G$ is a tree, $u$ is a vertex of degree at least 2 and $v \neq u$. Then, since the path from $v$ to $u$ is unique: one of the neighbors of $u$ is closer to $v$ and all the other ones are further away and thus
\begin{align*}
 \frac{1}{\mbox{deg}(u)} \sum_{(u, w) \in E} d(w,v) &=  \frac{1}{\mbox{deg}(u)} \left(d(u,v) - 1 + (\mbox{deg}(u) - 1) (d(u,v) + 1)\right)\\
 &\geq d(u,v)
 \end{align*}
and thus $u \not\in \partial G$.
\end{proof}

\subsection{Proof of Proposition 4}
\begin{proof}
We identify the $2d$ neighbors of $u \in V$ with $\pm e_1, \pm e_2, \dots, \pm e_d$, where $e_i = (0,0,\dots, 0, 1, 0, \dots, 0)$ is the $i-$th standard vector in $\mathbb{R}^d$. Suppose now that $u$ is a vertex in $G_{}$ which has $2d$ neighbors and that $u \in \partial G$ which means that there
exists $v \in V$ such that
$$ \frac{1}{2d} \sum_{(u, w) \in E} d(w,v) < d(u,v).$$
 We partition the set $\left\{1,2,\dots, d \right\}$ into two sets:
$$ A = \left\{ 1 \leq i \leq n: \exists \mbox{ shortest path from}~v~\mbox{to}~u~\mbox{containing}~u + e_i ~\mbox{or} ~u - e_i \right\}$$
and $B = \left\{1,2,\dots, n\right\} \setminus A$. For each $i \in B$, we can infer that
$$ d(u + e_i, v) \geq d(u,v) \qquad \mbox{and} \qquad d(u - e_i, v) \geq d(u,v).$$
If either of these inequalities are attained, then we have found two adjacent vertices with the same distance from $v$. Therefore we may assume that
$$ d(u + e_i, v) = d(u,v)+1 \qquad \mbox{and} \qquad d(u - e_i, v) = d(u,v)+1.$$
Let now $i \in A$: then there exists a shortest path through either $u + e_i$ or $u - e_i$ or both. If there exists one through both, then
$$ d(u \pm e_i, v) = d(u,v) - 1$$
and we are done. Let us thus assume that for each $i \in A$, exactly one of the two neighbors admits a shortest path. If the shortest
path goes through $u + e_i$, then $d(u - e_i, v) \in \left\{d(u,v), d(u,v)+1\right\}$. If it were $d(u,v)$, we would be done since we would have found two adjacent vertices with the same distance. Therefore we may
assume $d(u - e_i, v) = d(u,v) + 1$. In that case, we end up with
$$ \frac{1}{2d} \sum_{(u, w) \in E} d(w,v) = \frac{|A|}{2d} 2 d(u,v) + \frac{|B|}{2d} (2d(u,v) + 2) \geq d(u,v)$$
which contradicts $u \in \partial G$.
\end{proof}

\subsection{Proof of Theorem 2} 
\begin{proof}
Let $v_0 \in V$ be arbitrary. Denoting the maximal distance from $v_0$ by
$$ \ell = \max_{w \in V} d(v_0, w) \leq \diam(G),$$
we define for
$0 \leq i \leq \ell$, the sets
$$ A_i = \left\{ v \in V: d(v, v_0) = i \right\}$$
and note that these sets induce a partition of the vertices 
$$ \bigcup_{i=0}^{\ell} A_i = V.$$

\begin{center}
\begin{figure}[h!]
\begin{tikzpicture}[scale=1.25]
\filldraw (0,0) circle (0.04cm);
\draw[thick] (1,0) ellipse (0.3cm and 0.6cm);
\draw[thick] (2,0) ellipse (0.3cm and 0.6cm);
\node at (3,0) {$\dots$};
\draw[thick] (4,0) ellipse (0.3cm and 0.6cm);
\draw [thick] (0,0) -- (0.9, 0.2);
\draw [thick] (0,0) -- (0.9, -0.4);
\draw [thick] (0,0) -- (0.9, 0.5);
\draw [thick] (0.9, 0.5) -- (2.1, 0);
\draw [thick] (0.9, 0.5) -- (2.1, -0.3);
\draw [thick] (0.9, -0.4) -- (2.1, 0.2);
\draw [dashed] (3, 0.5) -- (2.1, 0);
\draw [dashed] (2.8, 0.5) -- (2.1, -0.3);
\draw [dashed] (2.9, -0.4) -- (2.1, 0.2);
\draw [dashed] (3, 0.5) -- (2.1, 0);
\draw [dashed] (2.8, 0.5) -- (2.1, -0.3);
\draw [dashed] (2.9, -0.4) -- (2.1, 0.2);
\node at (0, -0.85) {$A_0$};
\node at (1, -0.85) {$A_1$};
\node at (2, -0.85) {$A_2$};
\node at (4, -0.8) {$A_{\ell}$};
\draw [dashed] (4, 0.5) -- (3.1, 0);
\draw [dashed] (4, -0.2) -- (3.1, -0.3);
\draw [dashed] (4.2, 0.2) -- (3.1, 0.3);
\draw [dashed] (3.9, -0.4) -- (3.1, 0.2);
\end{tikzpicture}
\caption{Ordering vertices by distance from a fixed vertex.}
\end{figure}
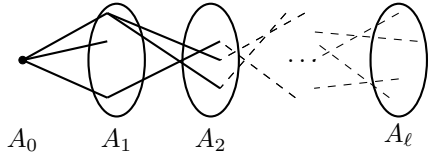
\end{center}

 A vertex $v \in A_i$ may or may not have a neighbor in $A_{i+1}$. We claim that if it does not, then $v \in \partial G$. This can be seen as follows: in the absence of neighbors in $A_{i+1}$ all of
 the neighbors of $v$ are either in $A_i$ or $A_{i-1}$. Moreover, since there is a shortest path from $v_0$ to $v$, it follows that $v$ has at least one neighbor in $A_{i-1}$ and 
 $$  \frac{1}{\deg(v)} \sum_{(v_0, w) \in E} d(w,v_0) \leq \frac{(d(w,v_0) -1) + (\deg(v) -1) d(w,v_0)}{\deg(v)} < d(w,v_0).$$
 Using again that each vertex in $A_i$ has at least one neighbor in $A_{i-1}$ implies for the number of edges between $A_{i-1}$ and $A_i$ that
$$ | E(A_{i-1}, A_{i})| \geq  | A_i |.$$
Let us now take a vertex $v \in A_i$ and suppose that $v \notin \partial G$ (more precisely that $v$ is not identified as being part of the boundary from the perspective of $v_0$). Then
 $$  \frac{1}{\deg(v)} \sum_{(v, w) \in E} d(w,v_0) \geq d(v,v_0).$$
Each neighbor $w$ of $v$ is either in $A_{i-1}$ or $A_i$ or $A_{i+1}$ and therefore has to satisfy
$$ d(w,v_0) \in \left\{ d(v_0, v) - 1, d(v_0,v), d(v_0,v) + 1\right\}.$$
Therefore, $v \in A_i$ not being identified as part of the boundary by $v_0$, implies that $v$ has at least as many neighbors in $A_{i+1}$ as in $A_i$, we have
$$ | \left\{w \in V: (v,w) \in E ~\quad~w \in A_{i+1} \right\} | \geq  | \left\{w \in V: (v,w) \in E ~ \mbox{and} ~w \in A_{i-1} \right\} |.$$
If $v \in A_i$ is identified as being part of the boundary by $v_0$, then we know that
$$  | \left\{w \in V: (v,w) \in E ~ \mbox{and} ~w \in A_{i-1} \right\} | \leq \Delta.$$
This leads to an interesting dichotomy: if $ | E(A_{i-1}, A_{i})| $ is large, if there are many edges from $A_{i-1}$ to $A_i$, then
there are either also many edges from $A_i$ to $A_{i+1}$ or there are many boundary points in $A_i$.
 Formally, 
$$  |E(A_{i-1}, A_{i})| \leq  | E(A_{i}, A_{i+1})| + \Delta \cdot | \partial G \cap A_i|.$$
This can be rewritten as
\begin{equation} \label{useful}
   |\partial G \cap A_i| \geq \frac{1}{\Delta} \left(| E(A_{i-1}, A_{i})| -  | E(A_{i}, A_{i+1})|\right).
   \end{equation}

\begin{center}
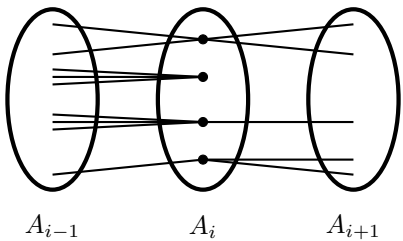
\begin{figure}[h!]
\begin{tikzpicture}[scale=2]
\draw[ultra thick] (1,0) ellipse (0.3cm and 0.6cm);
\draw[ultra thick] (2,0) ellipse (0.3cm and 0.6cm);
\draw[ultra thick] (3,0) ellipse (0.3cm and 0.6cm);
\node at (1, -0.85) {$A_{i-1}$};
\node at (2, -0.85) {$A_i$};
\node at (3, -0.85) {$A_{i+1}$};
\filldraw (2, 0.4) circle (0.03cm);
\filldraw (2, 0.15) circle (0.03cm);
\filldraw (2, -0.15) circle (0.03cm);
\filldraw (2, -0.4) circle (0.03cm);
\draw [thick] (1, 0.5) -- (2, 0.4) -- (3, 0.5);
\draw [thick] (1, 0.3) -- (2, 0.4) -- (3, 0.3);
\draw [thick] (1, 0.2) -- (2, 0.15);
\draw [thick] (1, 0.15) -- (2, 0.15);
\draw [thick] (1, 0.1) -- (2, 0.15);
\draw [thick] (1, -0.2) -- (2, -0.15);
\draw [thick] (1, -0.15) -- (2, -0.15);
\draw [thick] (1, -0.1) -- (2, -0.15) -- (3, -0.15);
\draw [thick] (1, -0.5) -- (2, -0.4) -- (3, -0.5);
\draw [thick] (2, -0.4) -- (3, -0.4);
\end{tikzpicture}
\caption{Each vertex $v \in A_i$ that is not in the boundary has at least as many neighbors in $A_{i+1}$ as in $A_{i-1}$. A boundary vertex $v \in A_i \cap \partial G$ absorbs at most  $\Delta$ incoming edges.}
\end{figure}
\end{center}

Note that each element in $A_{\ell}$ is necessarily in the boundary and thus
$$ |\partial G| \geq |A_{\ell}|.$$  
 However, it is certainly conceivable that $|A_{\ell}|$ is very small. In that case, there are at least $|V| - 1 - |A_{\ell}|$
 other vertices that are partitioned over the set $A_1, \dots, A_{\ell-1}$.
The pigeonhole principle implies that at least one of these sets is large: there exists an $1 \leq i_0 \leq \ell-1$ such that
$$ | A_{i_0}| \geq  \frac{|V|-1 - |A_{\ell}|}{\ell-1}.$$
This also implies, since each vertex in $A_{i_0}$ has at least one neighbor in $A_{i_0 -1}$,
$$ | E(A_{i_0 -1} , A_{i_0}) |\geq | A_{i_0}| \geq  \frac{|V|-1 - |A_{\ell}|}{\ell-1}.$$
Summing \eqref{useful} and using that
$$ |E(A_{\ell -1}, A_{\ell})| \leq \Delta \cdot |A_{\ell}|,$$
\begin{align*}
| \partial G| &\geq  \sum_{i = i_0}^{\ell - 1}|\partial G \cap A_i | \\
&\geq  \sum_{i = i_0}^{\ell - 1}  \frac{1}{\Delta} \left(| E(A_{i-1}, A_{i})| -  | E(A_{i}, A_{i+1})|\right) \\
&\geq \frac{1}{\Delta} \left( |E(A_{i_0 -1} , A_{i_0})| - | E(A_{\ell-1} , A_{\ell})|\right) \\
&\geq  \frac{1}{\Delta} \left(  \frac{|V|-1 - |A_{\ell}|}{\ell-1} -  |E(A_{\ell-1} , A_{\ell})|\right) \geq  \frac{1}{\Delta}   \frac{|V|-1 - |A_{\ell}|}{\ell-1} -  |A_{\ell}|. 
 \end{align*}
This implies
$$ |\partial G| \geq  \max\left\{|A_{\ell}|,\frac{1}{\Delta}   \frac{|V|-1 - |A_{\ell}|}{\ell-1} -  |A_{\ell}|  \right\}.$$
The maximum is minimized when
$$ |A_{\ell}| =  \frac{|V|-1}{2 \Delta (\ell-1) + 1}$$
and thus, using $\ell \leq \diam(G)$ and $\Delta \geq 1$,
$$ |\partial G|  \geq   \frac{|V|-1}{2 \Delta (\diam(G)-1) + 1} \geq  \frac{|V|-1}{2 \Delta \diam(G) -1}.$$
For positive $A,B > 0$, we have $(A-1)/(B-1) \geq A/B$ whenever $A \geq B$ and thus whenever $A/B \geq 1$. Combining this with $|\partial G| \geq 2$, we have
$$ |\partial G|  \geq  \frac{1}{2 \Delta }\frac{|V|}{\diam(G) }.$$

\end{proof}

\subsection{Proof of Proposition 5}

\begin{proof}  The argument is very short. We use
$$ \int_{\Omega}  \Delta d(x, y) dy \geq \int_{\Omega} \frac{d-1}{d(x,y)} dy \geq \frac{d-1}{\diam(\Omega)} |\Omega|.$$
Integration by parts gives
$$  \int_{\Omega}  \Delta d(x, y) dy = \int_{\partial \Omega} \frac{\partial d(x, z)}{\partial \nu} d \sigma(z).$$
The partial derivative is either positive (if the geodesic arrives nontangentially) or 0 (if the geodesic arrives tangentially). Therefore
$$  \int_{\partial \Omega} \frac{\partial d(x, z)}{\partial \nu} d \sigma(z) =  \int_{(\partial \Omega)_x} \frac{\partial d(x, z)}{\partial \nu} d \sigma(z).$$
Since the distance function is $1-$lipschitz, we have
$$ \int_{(\partial \Omega)_x} \frac{\partial d(x, z)}{\partial \nu} d \sigma(z) \leq  \int_{(\partial \Omega)_x}  d \sigma(z) = |(\partial \Omega)_x|.$$
\end{proof}

\end{document}